\documentclass[12pt]{article}

\usepackage[latin1]{inputenc}
\usepackage{fancyhdr}
\usepackage{indentfirst}
\usepackage{graphicx}
\usepackage{newlfont}

\usepackage{epsfig}
\usepackage{amssymb}
\usepackage{amsmath}
\usepackage{latexsym}
\usepackage{amsthm}
\usepackage{bbold}

\usepackage{mathrsfs}
\usepackage{esint}

\newtheorem{theorem}{Theorem}[section]
  \theoremstyle{plain}
\newtheorem{remark}[theorem]{Remark}
  \theoremstyle{plain}
\newtheorem{remark*}{Remark}
  \theoremstyle{plain}
\newtheorem{lemma}[theorem]{Lemma}
  \theoremstyle{plain}
\newtheorem*{lemma*}{Lemma}
  \theoremstyle{plain}
\newtheorem{proposition}[theorem]{Proposition}
  \theoremstyle{plain}
\newtheorem{corollary}[theorem]{Corollary}
  \theoremstyle{definition}
  \newtheorem{definition}[theorem]{\protect\definitionname}
  \theoremstyle{definition}
  \newtheorem*{definition*}{\protect\definitionname}
  \theoremstyle{plain}
  \newtheorem{assumption}[theorem]{\protect\assumptionname}

  \providecommand{\assumptionname}{Assumption}
  \providecommand{\definitionname}{Definition}

\begin{document}

\global\long\def\Rm{\mathbb{R}^{m}}
\global\long\def\N{\mathbb{N}}
\global\long\def\R{\mathbb{R}}
\global\long\def\P{\mathbb{P}}

\global\long\def\cD{\mathcal{D}}
\global\long\def\cM{\mathcal{M}}
\global\long\def\lebesgueL{\mathcal{L}}
\global\long\def\mupalm{\mu_{\mathcal{P}}}
 \global\long\def\borelB{\mathcal{B}}
\global\long\def\sF{\mathcal{F}}
\global\long\def\sE{\mathcal{E}}

\global\long\def\loc{\mathrm{loc}}
\global\long\def\pot{\mathrm{pot}}
\global\long\def\sol{\mathrm{sol}}

\global\long\def\Rmd{\mathbb{R}^{m\times m}}
\global\long\def\Zd{\mathbb{Z}^{d}}

\global\long\def\bU{\boldsymbol{\mathrm{U}}}

\global\long\def\cM{\mathcal{M}}
\global\long\def\cR{\mathcal{R}}
\global\long\def\cU{\mathcal{U}}
\global\long\def\cV{\mathcal{V}}

\global\long\def\sI{\mathscr{I}}

\global\long\def\rmD{\mathrm{D}}

\global\long\def\eps{\varepsilon}

\global\long\def\mugammapalm{\mu_{\Gamma,\mathcal{P}}}
\global\long\def\mupalm{\mu_{\mathcal{P}}}
\global\long\def\nupalm{\nu_{\mathcal{P}}}

\global\long\def\muomega{\mu_{\omega}}
\global\long\def\mugammaomega{\mu_{\Gamma(\omega)}}

\global\long\def\bQ{\boldsymbol{Q}}
\global\long\def\d{\mathrm{d}}

\global\long\def\norm#1{\left\Vert #1\right\Vert }

\global\long\def\ue{u^{\eps}}
\global\long\def\ve{v^{\eps}}

\global\long\def\weakto{\rightharpoonup}

\title{A mathematical model 
for {A}lzheimer's disease: An approach via
stochastic homogenization of the {S}moluchowski equation}

\author{Bruno Franchi\thanks{Dipartimento di Matematica, Universit\`a
di Bologna, Piazza di Porta S.~Donato 5, 40126 Bologna, Italy, 
(bruno.franchi@unibo.it).}, $\, $ 
Martin Heida\thanks{Weierstrass Institute for Applied Analysis and Stochastics, Mohrenstra{\ss}e 39, 10117 Berlin, Germany, (martin.heida@wias-berlin.de)}
$\, $ and $\, $ Silvia Lorenzani\thanks{Dipartimento di Matematica, Politecnico
di Milano, Piazza Leonardo da Vinci 32, 20133 Milano, Italy, 
(silvia.lorenzani@polimi.it).}} 

\date {}

\maketitle

\begin{abstract}
In this note, we apply the theory of stochastic homogenization to find the
asymptotic behavior of the solution of a set of Smoluchowski's 
coagulation-diffusion equations with non-homogeneous Neumann boundary
conditions.
This system is meant to model the aggregation and diffusion of 
$\beta$-amyloid peptide (A$\beta$) in the cerebral tissue, 
a process associated with
the development of Alzheimer's disease.
In contrast to the approach used in our previous works, in the present paper
we account for the non-periodicity of the cellular structure of the brain
by assuming a stochastic model for the spatial distribution of neurons.
Further, we consider non-periodic random diffusion coefficients for the
amyloid aggregates and a random production of A$\beta$ in the monomeric
form at the level of neuronal membranes.
\end{abstract}

\section{Introduction} \label{intro}

The primary feature of several neurological diseases, such as Prion diseases,
 Alzheimer's
disease, Parkinson's disease, Creutzfeldt-Jacob disease is the pathological 
presence of misfolded protein aggregates
(that is, proteins that fail to configure properly, becoming structurally
abnormal) \cite{CIE}, \cite{MFRL}. 
In this paper, we focus our interest in Alzheimer's
disease (AD). Indeed, AD has a huge social and economic impact. Until 2040 its worldwide global
prevalence (estimated as high as 44 millions in 2015) is expected to double every 20 years.
In particular, existing clinical data support the idea that amyloid-$\beta$
peptide (A$\beta$)  has a critical role as initiator
of a complex network of pathologic changes in the brain, ultimately leading
to Alzheimer's disease ('amyloid hypothesis', see e.g. \cite{Haass},
\cite{Karran}, \cite{Musiek}).
Although there is no doubt that the presence of fibrillar A$\beta$ deposition
(senile plaques) is the hallmark of the clinical syndrome of AD, the bulk of
human biomarker data reveals the existence of a discrepancy between
the appearance of amyloid deposits and clinical dementia, with A$\beta$
plaques anatomically disconnected from areas of severe neuronal loss.
One of the most reliable explanations, which also supports the amyloid
hypothesis, is that, in addition to fibrillar plaques, oligomeric forms of
A$\beta$ can play a dominant role in triggering a wide variety of pathogenic
effects.
Mice which accumulate A$\beta$ oligomers, but not fibrillar plaques, develop
synaptic damage, inflammation and cognitive impairment 
\cite{Tomiyama}, \cite{Zhang}.
Despite the biological relevance of the negative effects produced, the
exact mechanisms of misfolded protein aggregation and propagation, as well as
their toxicity, are still not well understood.
Furthermore, the complexity of the underlying processes makes it difficult
to extrapolate the effects of protein misfolding from the microscopic
(e.g. molecular) to the macroscopic (e.g. organs) scale, preventing the
development of effective therapeutic interventions.
In order to complement the medical and biological research, the last few
decades have seen the emergence of several mathematical models that can help
to provide a better insight into the laws governing the processes of protein
aggregation and the effects of toxicity spreading.
The mathematical approaches considered so far can be predominantly divided
into two different classes: on the one hand, there are the models designed
to describe processes at the molecular (microscopic) scale (aggregation
kinetics, short-range spatial spreading, etc...) 
\cite{CIE}, \cite{Achdou}, \cite{FL},
\cite{Helal} while, on the other side,
there are models that account for large-scale events characterizing the
progression of neurodegenerative misfolded protein-related diseases
\cite{CIE}, \cite{Bertsch}, \cite{BFTT}, \cite{Tosin}.

\subsection{A mathematical model for the aggregation of
$\beta$-amyloid based on Smoluchowski's equations.} \label{sec:mod-smol}
In $2013$, Achdou et al. proposed in \cite{Achdou} a mathematical model for the
aggregation and diffusion of $\beta$-amyloid peptide (A$\beta$) in the brain
affected by Alzheimer's disease (AD) at a microscopic scale (the size of a
single neuron).
In particular, these authors considered a portion of cerebral tissue,
represented by a bounded smooth region $\bQ \subset \mathbb{R}^3$, and described
the neurons as a family of regular regions $G_j$ such that:

\par\indent
(i) $\overline{G}_j \subset \bQ$ \, \, if $j=1,2, \ldots \overline{M}$;

\par\indent
(ii) $\overline{G}_i \cap \overline{G}_j =\varnothing$ \, \, if
$i \neq j$.

Then, the following system of Smoluchowski equations has been introduced:

\begin{eqnarray} \label{intro.1.1}
\begin{cases}
\frac{\displaystyle \partial{u_1}}{\displaystyle \partial t}(t,x)-d_1 \,
\triangle_x u_1(t,x)+u_1(t,x)&\sum_{j=1}^M a_{1,j}
u_j(t,x)= 0 
\\

\\
\frac{\displaystyle \partial u_1}{\displaystyle \partial \nu}
\equiv \nabla_x u_1 \cdot n=0
& \text{on } \partial \bQ \\

\\
\frac{\displaystyle \partial u_1}{\displaystyle \partial \nu}
\equiv \nabla_x u_1 \cdot n=
\eta_j
& \text{on } \Gamma_j,  \; \; j=1, \ldots , \overline {M} \\

\\
u_1(0,x)=U_1 \geq 0 
\end{cases}
\end{eqnarray}

if $1 < m < M$

\begin{eqnarray} \label{intro.1.2}
\begin{cases}
\frac{\displaystyle \partial{u_m}}{\displaystyle \partial t}(t,x)-d_m \,
\triangle_x u_m(t,x)+u_m(t,x) &\sum_{j=1}^M a_{m,j}
u_j(t,x)= \\
&\frac{\displaystyle 1}{\displaystyle 2} \sum_{j=1}^{m-1} a_{j,m-j} u_j u_{m-j}
\\

\\
\frac{\displaystyle \partial u_m}{\displaystyle \partial \nu}
\equiv \nabla_x u_m \cdot n=0 
& \text{on }  \partial \bQ  \\

\\
\frac{\displaystyle \partial u_m}{\displaystyle \partial \nu}
\equiv \nabla_x u_m \cdot n=0
& \text{on }  \Gamma_j, \, \; j=1, \ldots , \overline {M}\\

\\
u_m(0,x)=0 
\end{cases}
\end{eqnarray}

and

\begin{eqnarray} \label{intro.1.3}
\begin{cases}
\frac{\displaystyle \partial{u_M}}{\displaystyle \partial t}(t,x)-d_M \,
\triangle_x u_M(t,x)
=\frac{\displaystyle 1}{\displaystyle 2} 
\sum_{\substack{j+k \geq M \\ k< M  (\text{if } j=M) \\ 
j<M  (\text{if } k=M)}} &a_{j,k} 
\, u_j \,  u_{k} 
\\

\\
\frac{\displaystyle \partial u_M}{\displaystyle \partial \nu}
\equiv \nabla_x u_M \cdot n=0
& \text{on } \partial \bQ \\

\\
\frac{\displaystyle \partial u_M}{\displaystyle \partial \nu}
\equiv \nabla_x u_M \cdot n=0
& \text{on }   \Gamma_j, \; \; j=1, \ldots , \overline {M}\\

\\
u_M(0,x)=0  
\end{cases}
\end{eqnarray}
where $u_j (t,x)$ ($1 \leq j < M-1$) is the molar concentration at the point
$x$ and at the time $t$ of an A$\beta$ assembly of $j$ monomers, while
$u_M$ takes into account aggregations of more than $M-1$ monomers.
The production of A$\beta$ in monomeric form from the neuron membranes
has been modeled by coupling the evolution equation for $u_1$ with a
non-homogeneous Neumann condition on the boundaries of the sets $G_j$, 
indicated by $\Gamma_j$.
In particular, $0 \leq \eta_j \leq 1$ is a smooth function for $j=1, \ldots ,
\overline {M}$.
These monomers, by binary coalescence, give rise to larger assemblies,
which can diffuse in the cerebral tissue with a diffusion coefficient
$d_j$ that depends on their size.
Let us remark that, Achdou's model has been formulated to be valid on small
spatial domains, therefore isotropic diffusion has been assumed.  
The coagulation rates $a_{i,j}$ are symmetric $a_{i,j}=a_{j,i} >0$,
$i,j=1, \ldots ,M$, but $a_{MM}=0$, since it is assumed that long fibrils,
characterized by a very slow diffusion, do not coagulate with each other.

\subsection{Stochastic homogenization.} \label{sec:stoc-homo}
Since the development of modern imaging techniques (useful to evaluate the
progression of Alzheimer's disease) requires the need to test the predictions
of mathematical modeling at the macroscale, in the present paper, we have
applied the homogenization method to the model presented by Achdou et al.
\cite{Achdou}, in order to describe the effects of the production and 
agglomeration of the A$\beta$ at the macroscopic level.
The homogenization theory, introduced by the mathematicians in the seventies
to perform a sort of averaging procedure on the solutions of partial 
differential equations with rapidly varying coefficients or describing media
with microstructures, has been already successfully applied in \cite{FL},
\cite{FL_wheeden} to
derive a limiting model from that proposed by Achdou et al. \cite{Achdou}, in
the context of a periodically perforated domain.
In particular, in \cite{FL}, \cite{FL_wheeden} we have constructed our set 
${\bQ}^{\varepsilon}$,
starting from a fixed bounded domain ${\bQ}$ (which represents a portion of
cerebral tissue) and removing from it many small holes of characteristic
size $\varepsilon$ (the neurons) distributed periodically.
Then, we have rewritten the model problem (\ref{intro.1.1})-(\ref{intro.1.3}) 
as a family of
equations in ${\bQ}^{\varepsilon}$ and we have performed the limit
$\varepsilon \rightarrow 0$ in the framework of the two-scale convergence,
first introduced by Nguetseng \cite{Nguetseng} and Allaire \cite{Allaire}.
The peculiarity of the two-scale convergence method, used in 
\cite{FL}, \cite{FL_wheeden} to study
the limiting behavior of the Smoluchowski-type equations, is that in a single
process, one can find the homogenized equations and prove the convergence of
a sequence of solutions to the problem at hand.
Since the picture presented in our previous works 
\cite{FL}, \cite{FL_wheeden} is a too
crude oversimplification of the biomedical reality, in the present paper we
have chosen to resort to a stochastic parametrization of the model equations:
that is, we account for the non-periodic cellular structure of the brain.
In particular, the distribution of neurons is modeled in the following way:
it exists a family of predominantly genetic causes, not wholly deterministic,
which influences the position of neurons and the microscopic structure of the
parenchyma in a portion of the brain tissue ${\bQ}$.
Also, we consider non-periodic random diffusion coefficients and a random
production of A$\beta$ in the monomeric form at the level of neuronal
membranes.
This together defines a probability space $(\Omega, \mathcal{F}, \mathbb{P})$. 

Denoting by $\omega\in\Omega$ the random variable in our model, the 
set of random holes 
in $\Rm$ (representing the neurons) is labeled by $G(\omega)$.  
The production of $\beta$-amyloid at the boundary
$\Gamma(\omega)$ of $G(\omega)$ is
described by a random scalar function $\eta (x,\omega)$ and the
diffusivity, in the brain parenchyma, of clusters of different sizes 
$s$ is modeled by
random matrices $D_s(x,\omega)$ on $\Omega$.
For technical reasons, we assume that the randomness of the medium is 
stationary, 
that is, the probability distribution of the random variables observed 
in a set $A\subset\Rm$ is shift invariant (
all variables share the same distribution in $A$ and 
$A+x$, $x\in\Rm$). 
As shown by Papanicolaou and Varadhan \cite{PapVar} 
(who introduced this concept),
the assumption of stationarity provides a family of mappings 
$\left(\tau_x\right)_{x\in\Rm}: \Omega \rightarrow \Omega $ such that 
$\eta (x,\omega)=\eta(\tau_x\omega)$ and $D_s(x,\omega)=D_s(\tau_x\omega)$.
The periodic homogenization can be recovered in this frame
considering $\Omega=[0,1)^m$ 
with $\tau_x\omega=x+\omega\mod [0,1)^m$, where one  canonically chooses 
$\omega=0$ 
(see also \cite{Hei11}).

The above mentioned findings can  be interpreted in the sense that the 
stationarity of the coefficients and the resulting dynamical system  
$\tau_x$ transfer some  structural properties from $\Rm$ to $\Omega$ such 
that we could formally identify $\Omega\approx\Rm$. 
Accordingly, a stationary random set in $\Rm$ corresponds to a subset of $\Omega$ 
and a random Hausdorff measure on $\Rm$ corresponds to a measure on $\Omega$.
In order to prevent confusion, let us note that, all the 
similarities we mention here are of algebraic and measure-theoretic 
nature and not in the sense of a vector space isomorphism. 
With the above  short overview, we just  want to point out that many useful 
tools 
in periodic homogenization find their counterpart in the stochastic setting. 
The stochastic  homogenization theorems can  be formulated in a very
similar way to their 
periodic version, if we rely on the above connections and similarities, 
though the mathematics behind differs sometimes significantly.
In this framework, we have studied the limiting behavior of the system of
nonlinear Smoluchowski-type equations describing our model by using a sort
of stochastic version of the two-scale convergence method.

The rest of the paper is organized as follows.
In Section \ref{sec:rand-med}, we give a brief survey of the probabilistic
background behind the theory of stochastic homogenization and in Section
\ref{sec:two-scale} we present all the main definitions and theorems related
to the stochastic two-scale convergence method.
In Section \ref{sett:prob} we give a detailed description of our model and
derive all the a priori estimates needed to apply the two-scale homogenization
technique.
Then, Section \ref{sec:homo} is devoted to the presentation and the proof
of our main results on the stochastic homogenization of the nonlinear
Smoluchowski coagulation-diffusion equations in a randomly perforated
domain.
Finally,  Appendix \ref{appA} is introduced to summarize some basic concepts
on the realization of random sets.

\vskip2mm

\section{Random media} \label{sec:rand-med}

The method of stochastic two-scale convergence introduced by 
Zhikov and Piatnitsky \cite{ZP}
is based on a setting that was originally introduced by 
Papanicolaou and Varadhan \cite{PapVar}.
The connection between the abstract setting on random singular measures in 
\cite{ZP} 
and the theory of random sets was worked out in \cite{Hei11}. 
Hence we will first introduce the setting of \cite{PapVar} and 
explain the ideas pointed out in \cite{Hei11} before we move on to the 
definition of two-scale convergence.

\subsection{Stationary ergodic dynamical systems.} \label{sec:StErg-dyn-sys} 
This section has the intention to provide a probabilistic background for the 
theory of stochastic homogenization, 
and more particularly for stochastic two-scale convergence. 
We follow the formulation given by
Papanicolaou and Varadhan \cite{PapVar},
enriched by the ideas presented in \cite{JKO}, \cite{ZP} and \cite{Hei11,Hei17}.

The whole theory is based on the concept of dynamical systems.

\begin{definition}[Dynamical system] \label{def:dyn-sys}
Let $(\Omega, \mathcal{F}, \mathbb{P})$ be a probability space.
An $m$-dimensional dynamical system is defined as a family of 
measurable bijective mappings
$\tau_x: \Omega \rightarrow \Omega$, $x \in \mathbb{R}^m,$
satisfying the following conditions:

\par\indent
(i) the group property: $\tau_0=\mathbb{1}$ ($\mathbb{1}$ is the identity
mapping), $\tau_{x+y}=\tau_x \circ \tau_y \; \; \; 
\forall x,y \in \mathbb{R}^m$;

\par\indent
(ii) the mappings $\tau_x: \Omega \rightarrow \Omega$ preserve the measure
$\mathbb{P}$ on $\Omega$, i.e., for every $x \in \mathbb{R}^m$, and every
$\mathbb{P}$-measurable set $F \in \mathcal{F}$, we have
$\mathbb{P} (\tau_x F)=\mathbb{P} (F)$;

\par\indent
(iii) the map $\mathcal T: \Omega \times \mathbb{R}^m \rightarrow \Omega$:
$(\omega, x) \mapsto \tau_x \omega$ is measurable (for the standard 
$\sigma$-algebra on the product space, where on $\mathbb{R}^m$ we take
the Borel $\sigma$-algebra).
\end{definition}

\par\noindent
Note that (i) and (iii) imply that, for every $x\in\R^m$ and measurable 
$F\subset\Omega$, the set $\tau_x F$ is
measurable: since $\tau_{-x}\left(\tau_x F\right)=F$ 
we find that $\tau_x F$ is the projection of 
${\mathcal T}^{-1}(F)\cap \{-x\}\times\Omega$ onto $\Omega$. 
We define the notion of ergodicity for the dynamical system.

\begin{definition}[Ergodicity] \label{def:ergodic}
A dynamical system is called ergodic if one of the following equivalent
conditions is fulfilled

\par\indent
(i) given a measurable and invariant function $f$ in $\Omega$, that is
$$\forall x \in \mathbb{R}^m\quad f(\omega)=f(\tau_x \omega) $$
almost everywhere in $\Omega$, then
$$f(\omega)=\text{const. } \; \; \; \text{for } \; \; \mathbb{P}-a.e. \; \; 
\omega \in \Omega;$$

\par\indent
(ii) if $F \in \mathcal{F}$ is such that $\tau_x F=F \; \; \; 
\forall x \in \mathbb{R}^m$, then $\mathbb{P} (F)=0$ or $\mathbb{P} (F)=1$. 
\end{definition}

\begin{definition}[Stationarity] \label{d2.3}
Given a probability space $(\Omega, \mathcal{F}, \mathbb{P})$, a real valued
process is a measurable function $f: \mathbb{R}^m \times \Omega \rightarrow
\mathbb{R}$.
We will say $f$ is stationary if the distribution of the random variable
$f(y, \cdot): \Omega \rightarrow \mathbb{R}$ is independent of $y$, i.e., for
all $a \in \mathbb{R}$, $\mathbb{P} (\{ \omega: f(y, \omega) > a \})$ is
independent of $y$.
This is qualified by assuming the existence
of a dynamical system $\tau_y: \Omega \rightarrow \Omega$ ($y \in \mathbb{R}^m$)
 and saying that
$f: \mathbb{R}^m \times \Omega \rightarrow \mathbb{R}$ is stationary if
$$f(y+y', \omega)=f(y, \tau_{y'} \omega) \; \; \; \text {for all} \; \; 
y, y' \in \mathbb{R}^m \; \; \text {and } \; \; \omega \in \Omega.$$
\end{definition}
Finally, we say that a random variable $f: \mathbb{R}^m \times \Omega 
\rightarrow \mathbb{R}$ is stationary ergodic if it is stationary and the
underlying dynamical system is ergodic.
Naturally, if $f$ is taking values in a finite dimensional space, we will
say it is stationary if all of its components in a given basis
are stationary with respect to the same dynamical system. 
This property is also called jointly stationary.

\begin{remark}\cite{PapVar} \label{r2.1}
A function $f$ is stationary ergodic if and only if there is some
measurable function $\tilde {f}: \Omega \rightarrow \mathbb{R}$ such that
$$f(x, \omega)=\tilde {f}(\tau_x \omega).$$
For a fixed $\omega \in \Omega$ the function 
$x\mapsto \tilde {f}(\tau_x \omega)$ of
argument $x \in \mathbb{R}^m$ is said to be a realization of function
$\tilde {f}$.
\end{remark}

Let $L^{p} (\Omega)$ ($1\leq p<\infty$) denote the space formed
by (the equivalence classes of) measurable functions that are
$\mathbb{P}$-integrable with exponent $p$ and $L^{\infty} (\Omega)$
be the space of measurable essentially bounded functions.
If $f \in L^{p} (\Omega)$, then $\mathbb{P}$-almost all realizations
$f(\tau_x \omega)$ belong to $L^{p}_{\text{loc}} (\mathbb{R}^m)$
\cite{JKO}.

We define the following $m$-parameter group of operators in the space
$L^2 (\Omega)$:
$$U(x): L^2 (\Omega) \rightarrow L^2 (\Omega)\,,\qquad 
f\,\mapsto\, \left[U(x) f\right] (\omega):=f(\tau_x \omega)\,.$$
It is known \cite{JKO} that the
operator $U(x)$
is unitary for each $x \in \mathbb{R}^m$ and the group $U(x)$ is strongly 
continuous , i.e.
$$\forall f \in L^2 (\Omega)\,:\qquad 
\lim_{x \to 0} \Vert U(x) f-f \Vert_{L^2 (\Omega)}=0\,.$$
For $x=\{0, 0, \ldots, x_i, 0, \ldots, 0 \}$ we obtain a one-parameter
group whose infinitesimal generator will be denoted by $\rmD_i$ with domain 
$\cD_i(\Omega)$.
The unitarity of the group $U(x)$ implies that the operators $\rmD_i$
are skew-symmetric:
\begin{equation} \label{1.1}
\forall f,g\in\cD_i (\Omega)\,:\qquad\int_{\Omega} (\rmD_i f) \, g \, 
d\mathbb{P}=
- \int_{\Omega} f \, (\rmD_i g) \, d\mathbb{P}\,,
\end{equation}
and by definition of the generators we have
\begin{equation} \label{1.2}
\rmD_i f=\lim_{x_i \neq 0, \,  x_i \to 0} 
\frac{f(\tau_{x_i} \omega)-f(\omega)}{x_i}
\end{equation}
in the sense of convergence in $L^2 (\Omega)$. 
As Papanicolaou and Varadhan \cite{PapVar} have shown, 
almost every realization possesses a weak derivative and it holds
$$\frac{\partial}{\partial x_i} f(\tau_x \omega)=(\rmD_i f)(\tau_x \omega) 
\in L^2_{\text{loc}} (\mathbb{R}^m) \,.$$
Also we have that ($i \, \rmD_1, \ldots, i \, \rmD_m$) are commuting,
self-adjoint, closed, and densely defined linear operators on
$L^2 (\Omega)$ \cite{HPV}, and we may define 
$$ \rmD_\omega f:=\left(\rmD_1 f,\dots,\rmD_m f\right)^\top\,.$$

We introduce the space $W^{1,2}(\Omega)$ with norm $\|\cdot\|_{1,2}$ through 
\begin{align*}
W^{1,2}(\Omega)&:= {\mathcal D}_1 (\Omega) \bigcap \ldots \bigcap
{\mathcal D}_m (\Omega)\\
\|f\|_{1,2}&:=\|f\|_{L^2(\Omega)}+\sum_{i=1}^{m}\|\rmD_i f\|_{L^2(\Omega)}\,.
\end{align*}
Further let $L_{\loc}^{2}(\Rm;\Rm)$ be the set of measurable functions
$f:\,\Rm\to\Rm$ such that $f|_{\bU}\in L^{2}(\bU;\Rm)$ for every
bounded domain $\bU$ and we define 
\begin{align*}
L_{\pot,\loc}^{2}(\Rm) & :=\left\{ f\in L_{\loc}^{2}(\Rm;\Rm)\,\,|\,\,
\forall\bU\,\,\mbox{bounded domain, }\exists\varphi\in H^{1}(\bU)\,:\,
f=\nabla\varphi\right\} \,,\\
L_{\sol,\loc}^{2}(\Rm) & :=\left\{ f\in L_{\loc}^{2}(\Rm;\Rm)\,\,|\,\,
\int_{\Rm}f\cdot\nabla\varphi=0\,\,\forall\varphi\in C_{c}^{1}(\Rm)\right\} \,.
\end{align*}
Recalling the notion of a realization $f_{\omega}(x):=f(\tau_{x}\omega)$
for $f\in L^{2}(\Omega)$, we can then define corresponding spaces
on $\Omega$ through
\begin{align}
L_{\pot}^{2}(\Omega) & :=\left\{ f\in L^{2}(\Omega;\Rm)\,:\,
f_{\omega}\in L_{\pot,\loc}^{2}(\Rm)\,\,\mbox{for }\P-\mbox{a.e. }
\omega\in\Omega\right\} \,,\nonumber \\
L_{\sol}^{2}(\Omega) & :=\left\{ f\in L^{2}(\Omega;\Rm)\,:
\,f_{\omega}\in L_{\sol,\loc}^{2}(\Rm)\,\,\mbox{for }\P-\mbox{a.e. }
\omega\in\Omega\right\} \,,\label{eq:sto-Lp-pot-sol-omega-general}\\
\cV_{\pot}^{2}(\Omega) & :=\left\{ f\in L_{\pot}^{2}(\Omega)\,:\,
\int_{\Omega}f\,d\P=0\right\} \,.\nonumber 
\end{align}
It has been shown in Chapter 7 of \cite{JKO} that all of these spaces are 
closed and that $L^{2}(\Omega;\Rm)=L_{\sol}^{2}(\Omega)\oplus\cV_{\pot}^{2}
(\Omega)$.
This has been proved using the continuous smoothing operator 
\begin{equation}\label{eq:def-smoothing-omega}
\sI_{\delta}:\,L^{2}(\Omega)\to W^{1,2}(\Omega)\,,\qquad\sI_{\delta}
f(\omega):=\int_{\Rm}\eta_{\delta}(x)f\left(\tau_{x}\omega\right)\d x\,,
\end{equation}
where $\eta_{\delta}$ is a Dirac-sequence of smooth functions. It
can be shown that, for every $f\in L^{2}(\Omega)$, it holds 
$\sI_{\delta}f\to f$
as $\delta\to0$  and the continuity of $\sI_{\delta}$ implies 
$\rmD_{i}\sI_{\delta}f=\sI_{\delta}\rmD_{i}f$
for all $f\in W^{1,2}(\Omega)$. Thus, if we consider 
\[
\tilde{\cV}:=\mathrm{closure}_{L^{2}(\Omega)}\left\{ \rmD_{\omega}f\,:\;
f\in W^{1,2}(\Omega)\right\} 
\]
we first obtain $\tilde{\cV}\subseteq\cV_{\pot}^{2}(\Omega)$ and
for $g\in\tilde{\cV}^{\bot}$, we have for every $\delta>0$ 
\[
\forall f\in W^{1,2}(\Omega)\,:\quad0=\left\langle g,\rmD_{\omega}\sI_{\delta}
f\right\rangle =\left\langle \sI_{\delta}g,\rmD_{\omega}f\right\rangle 
=-\sum_{i}\left\langle \rmD_{i}\sI_{\delta}g,f\right\rangle \,,
\]
and hence $\sum_{i}\rmD_{i}\sI_{\delta}g=0$. 
In particular, $\sI_{\delta}g\in L_{\sol}^{2}(\Omega)$
and since $L_{\sol}^{2}(\Omega)$ is closed we find 
$\tilde{\cV}^{\bot}\subseteq L_{\sol}^{2}(\Omega)$.
This implies $\tilde{\cV}\supseteq\cV_{\pot}^{2}(\Omega)$ and hence
\begin{equation}\label{eq:char-l2pot}
\tilde{\cV}=\mathrm{closure}_{L^{2}(\Omega)}\left\{ \rmD_{\omega}f\,:\;
f\in W^{1,2}(\Omega)\right\}=\cV_{\pot}^{2}(\Omega)\,.
\end{equation} 

In what follows, we will often impose the following assumption:
\begin{assumption}
\label{assu:Omega-mu-tau-fundamental}Assume that $\Omega$ is a 
separable metric space and $(\Omega,\sF,\P)$
is a probability space with countably generated $\sigma$-algebra and let
$\tau_{x}$, $x\in\Rm$, be a  
dynamical system
in the sense of Definition \ref{def:dyn-sys} that
is ergodic in the sense of Definition \ref{def:ergodic}. 
\end{assumption}

It was discussed in \cite{Hei11} that the latter assumption 
is not a restriction to our choice of parameters.

By $\cM(\Rm)$ we denote the space of finitely bounded Borel measures
on $\Rm$ equipped with the Vague topology, which makes $\cM(\R^m)$ a separable 
metric space \cite{DV}. 
The $\sigma$-field defined by this topology is denoted by 
${\mathcal B}(\cM)$ since it is a Borel $\sigma$-field on $\cM$.
A random measure is a measurable mapping 
\[
\mu_{\bullet}:\;\Omega\to\cM(\Rm)\,,\qquad\omega\mapsto\mu_{\omega}
\]
which is equivalent to the measurability of all mappings 
$\omega\mapsto\mu_{\omega}(A)$,
where $A\subset\Rm$ are arbitrary bounded Borel sets. A random measure
is stationary if the distribution of $\mu_{\omega}(A)$ is invariant
under translations of $A$. 
In particular, random measures satisfy $\mu_{\tau_x \omega}(A)=\mu_{\omega}(A+x)$. 
For stationary random measures we find the following important property.

\begin{theorem}[\cite{DV} Existence of Palm measure and Campbell's Formula]
\label{theoremmecke} Let $\lebesgueL$ be the Lebesgue-measure on
$\Rm$ with $dx:=d\lebesgueL(x)$ and $(\Omega,\sF,\P)$ and $\tau$
as in Assumption \ref{assu:Omega-mu-tau-fundamental}. Then there exists
a unique measure $\mupalm$ on $\Omega$ such that 
\[
\int_{\Omega}\int_{\Rm}f(x,\tau_{x}\omega)\,\d\muomega(x)\d\P(\omega)=
\int_{\Rm}\int_{\Omega}f(x,\omega)\,\d\mupalm(\omega)\d x
\]
 for all $\borelB(\Rm)\times\borelB(\Omega)$-measurable non negative
functions and all $\mupalm\times\lebesgueL$- integrable functions.
Furthermore 
\begin{eqnarray}
\mupalm(A) & = & \int_{\Omega}\int_{\Rm}g(s)\chi_{A}(\tau_{s}\omega)
\d\muomega(s)\d\P(\omega)\,,\label{eq:defPalm-push-fw}\\
\int_{\Omega}f(\omega)\d\mupalm & = & \int_{\Omega}\int_{\Rm}g(s)f(\tau_{s}
\omega)\d\muomega(s)\d\P(\omega)
\end{eqnarray}
 for an arbitrary $g\in L^{1}(\Rm,\lebesgueL)$ with $\int_{\Rm}g(x)\d x=1$
and $\mupalm$ is $\sigma$-finite.
\end{theorem}
The measure $\mupalm$ from Theorem \ref{theoremmecke} is called Palm measure. 
By \eqref{eq:defPalm-push-fw} $\mupalm$ can be interpreted as the push-forward 
measure of $g(x)\d\muomega(x)\d\P(\omega)$ under 
$(x,\omega)\mapsto\tau_x\omega$. 
Stationarity implies that this push-forward is independent of the choice of $g$.
We say that the random measure $\mu_{\omega}$ has finite intensity if 
\begin{equation}
+\infty>\int_{\Omega}\int_{\Rm}\chi_{\Omega\times[0,1]^{m}}(\tau_{x}\omega,x)
\d\muomega(x)\,\d\P(\omega)=\mupalm(\Omega)\,.
\end{equation}

\begin{definition}\label{rescaling}
Given a stationary random measure $\mu_{\omega}$, we introduce the scaled
measure $\mu_{\omega}^{\varepsilon}$ through
\begin{equation}\label{rescaling 1}
\mu_{\omega}^{\varepsilon}(A):={\varepsilon}^m \, \mu_{\omega}
({\varepsilon}^{-1}
\, A).
\end{equation}
\end{definition}

One important property of random measures is the following generalization
of the Birkhoff ergodic theorem.

\begin{lemma}
\label{lem:ergodicity-cont-functions} (\cite{Hei17}, $Lemma\, 2.14$)  
Let Assumption
\ref{assu:Omega-mu-tau-fundamental} hold for $(\Omega,\sF,\P,\tau)$.
Let $\bQ\subset\Rm$ be a bounded domain, $\phi\in C(\overline{\bQ})$
and $f\in L^{1}(\Omega;\mupalm)$. Then, for almost every $\omega\in\Omega$
\begin{equation}
\lim_{\eps\rightarrow0}\int_{\bQ}\phi(x)\,f(\tau_{\frac{x}{\eps}}\omega)
\d\muomega^{\eps}(x)=\int_{\bQ}\int_{\Omega}\phi(x)f(\tilde\omega)
\d\mupalm(\tilde\omega)\,
\d x\,.\label{eq:ergodic-continuous}
\end{equation}

\end{lemma}

A further useful result towards this direction is the following.

\begin{lemma}
\label{lem:Ex-erg-repres} (\cite{Hei17}, $Lemma \,2.15$)  
Let Assumption \ref{assu:Omega-mu-tau-fundamental}
hold for $(\Omega,\sF,\P,\tau)$. Let $\bQ\subset\Rm$ be a bounded
domain and let $f\in L^{\infty}(\bQ\times\Omega;\lebesgueL\otimes\mupalm)$.
Then, $f$ has a $\borelB(\bQ)\otimes\sF$-measurable representative
which is an ergodic function in the sense that for almost every 
$\omega\in\Omega$
\begin{equation}
\begin{aligned}\lim_{\eps\rightarrow0}\int_{\bQ}
f(x,\tau_{\frac{x}{\eps}}\omega)\,\d\muomega^{\eps}(x) & 
=\int_{\bQ}\int_{\Omega}f(x,\tilde{\omega})\,\d\mupalm(\tilde{\omega})\,
\d x\,,\\
\lim_{\eps\rightarrow0}\int_{\bQ}\left|f(x,\tau_{\frac{x}{\eps}}\omega)
\right|^{p}\,\d\muomega^{\eps}(x) & =\int_{\bQ}\int_{\Omega}
\left|f(x,\tilde{\omega})\right|^{p}\,\d\mupalm(\tilde{\omega})\,\d x\,
\end{aligned}
\label{eq:ergodicity-general-functions}
\end{equation}
for every $1\leq p<\infty$. 
\end{lemma}

Based on the previous lemma, we can get the following result:
\begin{lemma}\label{lem:Ex-erg-repres-general}
Let Assumption \ref{assu:Omega-mu-tau-fundamental} hold for 
$(\Omega,\sF,\P,\tau)$.
Let $\bQ\subset\R^m$ be a bounded domain and let 
$f\in L^{\infty}(\bQ\times\Omega;\lebesgueL\otimes\mupalm)$.
Then, $f$ has a $\borelB(\bQ)\otimes\sF$-measurable representative
which is an ergodic function in the sense that for almost every 
$\omega\in\Omega$
and for all $\varphi\in C(\overline{\bQ})$ it holds 
\begin{equation}
\begin{aligned}\lim_{\eps\rightarrow0}
\int_{\bQ}f(x,\tau_{\frac{x}{\eps}}\omega)\varphi(x)\,\d\muomega^{\eps}(x) & =
\int_{\bQ}\int_{\Omega}f(x,\tilde{\omega})\varphi(x)\,
\d\mupalm(\tilde{\omega})\,\d x\,,\\
\lim_{\eps\rightarrow0}\int_{\bQ}\left|f(x,\tau_{\frac{x}{\eps}}\omega)
\right|^{p}\varphi(x)\,\d\muomega^{\eps}(x) & =
\int_{\bQ}\int_{\Omega}\left|f(x,\tilde{\omega})\right|^{p}
\varphi(x)\,\d\mupalm(\tilde{\omega})\,\d x
\end{aligned}
\label{eq:ergodicity-general-functions-test}
\end{equation}
for every $1\leq p<\infty$.
\end{lemma}
\begin{proof}
This follows from the fact that $C(\overline{\bQ})$ is separable
and Lemma \ref{lem:Ex-erg-repres} yields 
(\ref{eq:ergodicity-general-functions-test})
for a countable subset of $C(\overline{\bQ})$ and a set of full measure
$\tilde{\Omega}\subset\Omega$. 
By an approximation $\left\Vert \varphi-\varphi_{\delta}\right\Vert _{\infty}<
\delta$
and Lemma \ref{lem:Ex-erg-repres} we obtain the claim.
\end{proof}

\subsection{Random measures and random sets.}\label{sec:rand-meas-an-sets}
In this paper, we consider random sets of the following form. For
every $\omega\in\Omega$ the set $G(\omega)$ is an open
subset of $\Rm$. The boundary $\Gamma(\omega)=\partial G(\omega)$
is a $(m-1)$-dimensional piece-wise Lipschitz manifold. Furthermore,
we assume that the measures 
\[
\muomega(A):=\int_{A\cap G^\complement(\omega)}\d x\,,
\qquad\mugammaomega(A):=\mathcal{H}^{m-1}(A\cap\Gamma(\omega))
\]
are stationary. Hence, by Theorem \ref{theoremmecke} there exist
corresponding Palm measures $\mupalm$ for $\muomega$ and $\mugammapalm$
for $\mugammaomega$ and by Lemma 2.14 of \cite{Hei11}  there exists
a measurable set $\Gamma\subset\Omega$ with $\chi_{\Gamma(\omega)}(x)=
\chi_{\Gamma}(\tau_{x}\omega)$
for $\lebesgueL+\mugammaomega$-almost every $x$ for $\P$-almost
every $\omega$ and $\P(\Gamma)=0$, $\mugammapalm(\Omega\backslash\Gamma)=0$.
Also it was observed there that, if for every $\omega$ we have 
$\muomega=\lebesgueL$, then also $\mupalm=\P$. 
From the corresponding proofs in \cite{Hei11}, as well as the fact 
that $\muomega$
has a Radon-Nikodym derivative with respect to $\lebesgueL$, we find
$G\subset\Omega$ such that $\mupalm(A)=\P(A\cap G^\complement)$, 
$\chi_{G^\complement(\omega)}(x)=\chi_{G^\complement}(\tau_{x}\omega)$ and
\begin{equation}\label{eq:representation-muPalm}
    \chi_{G^\complement}\d \mupalm=d\mupalm=\chi_{G^\complement}\d\P\,.
\end{equation}
\begin{remark} If $A$ is a bounded Borel set, then
\begin{equation}\label{rescaling 2}
\mu_{\Gamma(\omega)}^{\varepsilon}(A):={\varepsilon}^m \,\mu_{\Gamma(\omega)}
({\varepsilon}^{-1}
\, A) = \varepsilon \mathcal{H}^{m-1}(A\cap \Gamma^\varepsilon(\omega)) .
\end{equation}

\end{remark}

It was shown in \cite{Hei11} that for random measures such as 
$\mu_\omega$ or $\mu_{\Gamma(\omega)}$ the underlying 
probability space can be assumed to be separable and metric, since
the boundedly finite Borel measures
equipped with the Vague topology form a separable metric space \cite{DV}.
It was also pointed out in \cite{Hei11} that $\tau: (x, \omega) \mapsto
\tau_x \omega$ is continuous.

\begin{remark} \label{rem:es-count-dense-set}
If $\Omega$ is separable and metric, this implies that $L^2(\Omega;\P)$ and 
$L^2(\Omega,\mugammapalm)$
are separable and that the bounded continuous functions $C_b(\Omega)$ are
dense in both spaces. Therefore, there exists a countable set 
$\Psi:=\left(\psi_i\right)_{i\in\N}$ such that $\psi_i\in C_b(\Omega)$  
for every $i$ and such that $\Psi$
lies dense in $L^2(\Omega;\P)$ and $L^2(\Omega,\mugammapalm)$. Furthermore,
recalling \eqref{eq:def-smoothing-omega} and approximating $\psi_i$ with the 
sequence $\sI_{\frac1n}\psi$, $n\in\N$, 
we can assume that 
$\psi_i\in W^{1,2}(\Omega)\cap C_b(\Omega)$.
The space $\cV^2_\pot (\Omega)$ is a subspace of a separable space and hence 
has to be separable, too. In particular $\nabla \psi_i$ can be assumed to be 
dense in $\cV^2_\pot (\Omega)$. We then define
$$\Psi=\left(\psi_i\right)_{i\in\N}\bigcup_{j=1}^m\left(\rmD_j
\psi_i\right)_{i\in\N}\,.$$
\end{remark}
Since $\Omega$ is assumed to be separable metric, we can also make the 
following definition.
\begin{definition}
The space of bounded continuously differentiable functions on $\Omega$ is 
\begin{align*}
C_b^1(\Omega)&:=\left\{f\in C_b(\Omega)\,:\;\rmD f\in C_b (\Omega) \right\}    \\
\norm{f}_{C^1_b(\Omega)}&:=\norm{f}_{\infty}+\norm{\rmD f}_{\infty}.
\end{align*}
\end{definition}
Let us remark that, since $(x, \omega) \mapsto \tau_x \omega$ is continuous,
$f \in C_b^1(\Omega)$ implies $f(\tau_x \omega) \in C_b^1(\Rm) \, \, \forall
\omega \in \Omega$.
Concerning the random geometries considered in this work, we make the 
assumptions listed below.

\begin{definition}[See \cite{GK}]\label{def:minimally-smooth}
An open set $G\subset\Rm$ is said to be minimally smooth with constants 
$(\delta,N,M)$ if we may cover $\Gamma=\partial G$ by a countable sequence of 
open sets $\left(U_i\right)_{i\in\N}$ such that
\begin{enumerate}
\item[1)] Each $x\in\Rm$ is contained in at most $N$ of the open sets $U_i$.
\item[2)] For any $x\in\Gamma$, the ball $B_\delta(x)$ is contained in at 
least one $U_i$.
\item[3)] For any $i$, the portion of the boundary $\Gamma$ inside $U_i$ agrees 
(in some Cartesian system of coordinates) with the graph of a Lipschitz 
function whose Lipschitz semi-norm is at most $M$.
\end{enumerate}
\end{definition}
In particular a set $G\subset\Rm$ is minimally smooth if and only if 
$\Rm\setminus G$ is minimally smooth.

Let $\bQ$ be a bounded domain in $\Rm$. For given constants $(\delta,N,M)$,
we consider $G(\omega)$ a random open set which is a.s. minimally smooth
with constants $(\delta,N,M)$ (uniformly minimally smooth).
We furthermore assume that $G(\omega):=\bigcup_{i\in\N}G_i(\omega)$
is a countable union of disjoint open balls $G_i(\omega)$ with a
maximal diameter $d_0$.

We then consider $G^\eps (\omega):=\eps G(\omega)$ and
\begin{equation}\label{eq:def-Q-eps-Gamma-eps}
  \bQ^\eps(\omega):=\bQ\backslash \left(\bigcup_{i\in I_\eps(\omega)} 
\eps G_i(\omega)\right)\,,\qquad 
  \Gamma_{\bQ}^\eps (\omega):=\bigcup_{i\in I_\eps(\omega)} 
\partial (\eps G_i(\omega))\,,
\end{equation}
where
$$ I_\eps(\omega):=\left\{i\,:\;\eps G_i(\omega)\subset\bQ\text{ and } 
\eps d_0<\min\left\{d(x,y)\,:\;x\in\partial (\eps G_i(\omega)),\,
y\in\partial\bQ\right\} \right\}\,.$$

\begin{remark}\label{7 feb:1}
Note that we constructed the micro structures $\bQ\backslash\bQ^\eps(\omega)$
such that they do not intersect
with the boundary of $\bQ$ and such that every hole in $\bQ^\eps(\omega)$
has a minimal distance $\eps d_0$ to $\partial\bQ$. This is because we require
in our proofs that $\eps^{-1}\bQ^\eps(\omega)$ is a 
$(\delta,N,M)$- minimal set
(or $\bQ^\eps(\omega)$ is a $(\delta\eps,N,\eps^{-1}M)$ minimal set, 
respectively).
In particular, without the minimal distance between two disjoint parts of the 
boundary,
the resulting set $\bQ^\eps(\omega)$ would violate condition $3)$ from 
Definition \ref{def:minimally-smooth},
i.e. $\partial \bQ^\eps(\omega)$ would not be a $\eps^{-1}M$-Lipschitz graph 
inside balls of diameter $\frac\eps2 d_0$.
\end{remark}

\begin{assumption} \label{1}
There are constants $d_0, \delta, N, M$ (independent of $\omega$) such that
$\mathbb{P}$-a.s.  the set $G(\omega)$ consists of a
countable union of bounded sets $G_k(\omega) \, \, (k \in \mathbb{N})$
such that the sets $\mathbb{R}^m \setminus G_k(\omega)$ are all connected, 
while
$$d(G_k(\omega), G_j(\omega)) \geq d_0 \; \; \; \text{ whenever } \, \,
k \ne j,$$
and each set $G_k(\omega)$ is minimally smooth with constants
($\delta, N, M$) and has a diameter smaller
than $d_0$. The Lipschitz constant is uniformly over all $G_k$.
\end{assumption}

\begin{remark}\label{7 feb:2}
In particular, this guarantees that $\Rm\setminus G(\omega)$ is connected 
and has
a Lipschitz boundary $\partial G$, which represents the union
of the boundaries of the holes. 
Furthermore, the distance condition ensures that 
the boundary of $G(\omega)$ is locally representable as a graph.
\end{remark}

\begin{lemma} \label{lem:extension-Op}
Suppose that Assumption \ref{1} is satisfied.
Then, there exists a family of linear continuous extension operators

$$\sE_\eps\,:\;W^{1,p}(\bQ^\eps)\to W^{1,p}(\bQ)$$ 
and a constant $C >0$ independent of $\varepsilon$ such that

$$\sE_\eps \phi=\phi \, \, \, \text{in} \, \, \bQ^\eps(\omega)$$
and

\begin{equation} \label{2.1}
\int_{\bQ} \vert \sE_\eps \phi \vert^p \, dx \leq C \, \int_{\bQ^\eps}
\vert \phi \vert^p \, dx,
\end{equation}
\begin{equation} \label{2.2}
\int_{\bQ} \vert \nabla (\sE_\eps \phi) \vert^p \, dx \leq C \, \int_{\bQ^\eps}
\vert \nabla \phi \vert^p \, dx,
\end{equation}
$\P$-a.s. for any $\phi \in W^{1,p}(\bQ^\eps)$ and for any $p \in (1, +\infty)$. 
\end{lemma}

\begin{proof}
Following the line of the proof reported in \cite{GK} (Proposition $3.3$,
p. $230$),
for any $k \in \mathbb N$, $\omega \in \Omega$, we denote by 
$\hat{G}_k (\omega)$ a $d_0/4$-neighborhood
of $G_k (\omega)$ (the sets $G_k (\omega)$ are defined in Assumption \ref{1}).
Since, under our assumptions, the set $\hat{G}_k (\omega) \backslash 
G_k (\omega)$ has
Lipschitz boundary, then, according to Theorem $5$, p. $181$ in \cite{Stein}, 
there exists
an extension operator $E_k$

\begin{equation} \label{2.3}  
E_k: W^{1,p} (\hat{G}_k (\omega) \backslash G_k (\omega)) \rightarrow 
W^{1,p} (\hat{G}_k (\omega))
\end{equation}
such that: $E_k \phi=\phi$ a.e. in $\hat{G}_k (\omega) \backslash G_k (\omega)$
and, for some constant $C$ independent of $k$, we have

\begin{equation} \label{2.4}
\Vert E_k \phi \Vert_{L^p(\hat{G}_k (\omega))} \leq C \, \Vert \phi 
\Vert_{L^p(\hat{G}_k (\omega) \backslash G_k (\omega))}
\end{equation}
\begin{equation} \label{2.5}
\Vert E_k \phi \Vert_{W^{1,p}(\hat{G}_k (\omega))} \leq C \, \Vert \phi
\Vert_{W^{1,p}(\hat{G}_k (\omega) \backslash G_k (\omega))}.
\end{equation}
Let us define new extensions

\begin{equation} \label{2.6}
\hat{E}_k: W^{1,p}(\hat{G}_k (\omega) \backslash G_k (\omega)) \rightarrow
W^{1,p}(G_k (\omega))
\end{equation}
by

\begin{equation} \label{2.7}
\hat{E}_k \phi:= E_k (\phi- (\phi)_k)+(\phi)_k
\end{equation}
where

\begin{equation} \label{2.8}
(\phi)_k:= \int_{\hat{G}_k (\omega) \backslash G_k (\omega)} \phi \, dy
\end{equation}
Putting them all together, we define an extension

\begin{equation} \label{2.9}
\sE: W^{1,p} (G^\complement(\omega)) \rightarrow W^{1,p} (\bQ)
\end{equation}
given by

\begin{equation} \label{2.10}
\sE \phi(y):= \left\{ \begin{array}{rl}
\phi(y) & \text{whenever}\, \,  y \in G^\complement (\omega)  \\
\hat{E}_k \phi(y) & \text{whenever}\, \, y \in \hat{G}_k (\omega).
\end{array} \right.
\end{equation}
Now, in $\hat{G}_k (\omega) \backslash G_k (\omega)$ we have

\begin{equation} \label{2.11}
\hat{E}_k \phi=(\phi-(\phi)_k)+(\phi)_k=\phi.
\end{equation}
Moreover, by (\ref{2.4}) and H\"older's inequality, we have

\begin{equation} \label{2.12}
\begin{split}
&\int_{G_k (\omega)} \vert \hat{E}_k \phi \vert^p \, dy=
\int_{G_k (\omega)} \vert E_k (\phi-(\phi)_k)+(\phi)_k \vert^p \, dy \\
&\leq C \, \int_{G_k (\omega)} \vert E_k (\phi-(\phi)_k) \vert^p \, dy+
C \, \int_{G_k (\omega)} \vert (\phi)_k \vert^p \, dy \\
&\leq C \, \int_{\hat{G}_k(\omega) \backslash G_k(\omega)} 
\vert \phi-(\phi)_k \vert^p \, dy
+C \, \int_{G_k (\omega)} \vert (\phi)_k \vert^p \, dy \\
&\leq C \, \int_{\hat{G}_k(\omega) \backslash G_k(\omega)} \vert \phi \vert^p \, dy+
 C' \, \int_{G_k (\omega)} \vert (\phi)_k \vert^p \, dy \\
&\leq C \, \int_{\hat{G}_k (\omega) \backslash G_k (\omega)} \vert \phi 
\vert^p \,dy
\end{split}
\end{equation}
where, for simplicity, the letter $C$ denotes a positive constant
(independent of $k$) that can change from line to line.
Due to Assumption \ref{1}, the following Poincar\'e inequality holds:

\begin{equation} \label{2.13}
\int_{\hat{G}_k (\omega) \backslash G_k (\omega)} \vert \phi-(\phi)_k \vert^p
\, dy \leq C \, \int_{\hat{G}_k (\omega) \backslash G_k (\omega)}
\vert \nabla \phi \vert^p \, dy
\end{equation}
Therefore, by using (\ref{2.5}), (\ref{2.7}) and (\ref{2.13}), we get

\begin{equation} \label{2.14}
\begin{split}
&\int_{G_k (\omega)} \vert \nabla (\hat{E}_k \phi) \vert^p \, dy=
\int_{G_k (\omega)} \vert \nabla (E_k (\phi-(\phi)_k)) \vert^p \, dy \\
& \leq C \, \int_{\hat{G}_k (\omega) \backslash G_k (\omega)}
\vert (\phi-(\phi)_k) \vert^p \, dy+ C \,
\int_{\hat{G}_k (\omega) \backslash G_k (\omega)} \vert \nabla \phi \vert^p
\, dy \\
&\leq C \, \int_{\hat{G}_k (\omega) \backslash G_k (\omega)} \vert
\nabla \phi \vert^p \, dy
\end{split}
\end{equation}
Since this holds for every $k$ with the same $C$ we have proved that

\begin{equation} \label{2.15}
\int_{\cup_k \hat{G}_k (\omega)} \vert \sE \phi \vert^p \, dy
\leq C \, \int_{\cup_k \hat{G}_k (\omega) \backslash G_k (\omega)}
\vert \phi \vert^p \, dy
\end{equation}

\begin{equation} \label{2.16}
\int_{\cup_k \hat{G}_k (\omega)} \vert \nabla (\sE \phi) \vert^p \, dy
\leq C \, \int_{\cup_k \hat{G}_k (\omega) \backslash G_k (\omega)}
\vert \nabla \phi \vert^p \, dy
\end{equation}
that is,

\begin{equation} \label{2.17}
\int_{\bQ} \vert \sE \phi \vert^p \, dy
\leq C \, \int_{G^\complement(\omega)}
\vert \phi \vert^p \, dy
\end{equation}

\begin{equation} \label{2.18}
\int_{\bQ} \vert \nabla (\sE \phi) \vert^p \, dy
\leq C \, \int_{G^\complement(\omega)}
\vert \nabla \phi \vert^p \, dy.
\end{equation}
By performing the change of variable $y=x/\varepsilon$, with $x \in \bQ^{\eps}
(\omega)$,
it is easy to obtain the
corresponding re-scaled estimates (\ref{2.1}) and (\ref{2.2}),
where $\sE_{\eps}$ is the re-scaled extension operator.

\end{proof}

As a matter of fact, we can describe a portion of the 
cerebral cortex
as a bounded open set $\bQ \subset \mathbb{R}^3$,
whereas the neurons are represented by a family of holes distributed randomly 
in $\bQ$ and having a characteristic size $\eps$.
A detailed construction of random domains that satisfy the assumptions 
listed in this section is reported in Appendix \ref{appA}.

\vskip2mm

\section{Two-scale convergence} \label{sec:two-scale}

We will use a slightly modified version of stochastic two-scale convergence
compared to the one presented in \cite{Hei17}. 
Let $\Psi:=\left(\psi_{i}\right)_{i\in\N}$
be the countable dense family of $C_{b}(\Omega)$-functions according
to Remark \ref{rem:es-count-dense-set}.  

\begin{lemma}\label{lem:f-i-g-i-ergodic}
Let $\left(f_i\right)_{i\in\N}$ be a countable family in 
$L^\infty(\bQ\times\Omega;\lebesgueL\times\P)$ and 
$\left(g_i\right)_{i\in\N}$ be a countable family in 
$L^\infty(\bQ\times\Gamma;\lebesgueL\times\mugammapalm)$. 
Then there exists a set of full measure $\Omega_{\Psi}\subset\Omega$ such that
for almost every $\omega\in\Omega_{\Psi}$, every $i\in\N$, every $\psi\in\Psi$ 
and every
$\varphi\in C_b(\overline{\bQ})$ the following holds:
\begin{equation} \label{dstar1}
\lim_{\eps\to0}\int_{\bQ}\varphi^{2}(x)\psi^{2}(\tau_{\frac{x}{\eps}}\omega)
f_i^2(x,\tau_{\frac{x}{\eps}}\omega)
\d x  =\int_{\bQ}\int_{\Omega}\varphi^{2}(x)\psi^{2}(\tilde\omega)f_i^2(x,\tilde\omega)\,\d\P(\tilde\omega)\,
\d x\,,
\end{equation}
\begin{equation} \label{dstar2}
\lim_{\eps\to0}\int_{\bQ}g_i\left(x,\tau_{\frac{x}{\eps}}\omega\right)\varphi(x)
\psi(\tau_{\frac{x}{\eps}}\omega)\d\mu_{\Gamma(\omega)}^{\eps}(x)  =
\int_{\bQ}\int_{\Omega}g_i(x,\tilde\omega)\varphi(x)\psi(\tilde\omega)\,
\d\mugammapalm(\tilde\omega)\,\d x\,.
\end{equation}
\end{lemma}
\begin{remark}
The first equality \eqref{dstar1} is needed for the proof of existence 
of the two-scale limits. Therefore we put the square here. 
The second limit \eqref{dstar2} is needed directly in the proof of the 
main homogenization theorem. 
Therefore we study the convergence of $g_i$ tested with $\varphi\psi$.
\end{remark}
\begin{proof}({\bf Proof of Lemma \ref{lem:f-i-g-i-ergodic}})
For fixed $i$ the limits \eqref{dstar1} and \eqref{dstar2}  
hold for a.e. $\omega\in\Omega$ due to Lemma \ref{lem:Ex-erg-repres-general}.
Since the family $\left(f_i\right)_{i\in\N}$ is countable, we conclude.\end{proof}

\begin{definition}
\label{def:two-scale-conv} Let $\Psi$ be the set of 
Remark \ref{rem:es-count-dense-set}
and let $\omega\in\Omega_{\Psi}$. Let $\ue\in L^{2}(\bQ)$ for all
$\eps>0$. We say that $(\ue)$ converges (weakly) in two scales to
$u\in L^{2}(\bQ;L^{2}(\Omega))$ and write $\ue\stackrel{2s}{\weakto}u$
if $\sup_{\eps>0}\norm{\ue}_{L^{2}(\bQ)}<\infty$ and if for every
$\psi\in\Psi$, $\varphi\in C(\overline{\bQ})$ there holds with 
$\phi_{\omega,\eps}(x):=\varphi(x)\psi(\tau_{\frac{x}{\eps}}\omega)$ that
\[
\lim_{\eps\to0}\int_{\bQ}\ue(x)\phi_{\omega,\eps}(x)\d x=
\int_{\bQ}\int_{\Omega}u(x,{\tilde \omega})\varphi(x)
\psi(\tilde \omega)\,\d\P(\tilde \omega)\,\d x\,.
\]
\end{definition}
Furthermore, we say that $\ue$ converges strongly in two scales to $u$,
written $\ue \stackrel{2s}{\to} u$, if for every weakly two-scale
converging sequence $v^{\varepsilon} \in L^2(\bQ)$ with
$v^{\varepsilon} \stackrel{2s}{\weakto} v \in L^2(\bQ;L^{2}(\Omega))$ as
$\varepsilon \rightarrow 0$ there holds
\begin{equation} \label{onestar}
\lim_{\eps\to0}\int_{\bQ}\ue v^{\varepsilon}\,\d x=
\int_{\bQ}\int_{\Omega}u \, v \,\d\P(\tilde \omega)\,\d x\,.
\end{equation}

\begin{remark}
Let us remark that the notion of two-scale convergence strongly depends on
the choice of $\omega$. 
Also, let us note that $\phi_{\omega,\eps}\stackrel{2s}{\to}\varphi \psi$
strongly in two scales by definition. 
\end{remark}

\begin{lemma}
(\cite{Hei17}, $Lemma \, 4.4$-$1$.) \label{lem:Existence-ts-lim} 
Let $\ue\in L^{2}(\bQ)$ be a sequence of functions such that 
$\norm{\ue}_{L^{2}(\bQ)}\leq C$
for some $C>0$ independent of $\eps$. Then there exists a subsequence
$(u^{\eps'})_{\eps'\to0}$ and $u\in L^{2}(\bQ;L^{2}(\Omega))$
such that $u^{\eps'}\stackrel{2s}{\weakto}u$ and 
\begin{equation}
\norm u_{L^{2}(\bQ;L^{2}(\Omega))}\leq\liminf_{\eps'\to0}
\norm{u^{\eps'}}_{L^{2}(\bQ)}\,.\label{eq:two-scale-limit-estimate}
\end{equation}
Furthermore, let $\left(f_i\right)_{i\in\N}$ 
be a family of functions such as in Lemma \ref{lem:f-i-g-i-ergodic}. 
Then for every $i\in\N$, $\varphi\in C(\overline\Omega)$ and 
$\psi\in\Psi$ it holds 
\begin{equation}\label{eq:ts-conv-conclusion}
\lim_{\eps\to0}\int_{\bQ}\ue(x)\phi_{\omega,\eps}(x)f_i(x,\tau_{\frac x\eps}
\omega)\d x=
\int_{\bQ}\int_{\Omega}u(x,\tilde \omega)\varphi(x)\psi(\tilde \omega)
f_i(x,\tilde \omega)\,
\d\P(\tilde \omega)\,\d x\,.
\end{equation}
\end{lemma}
\begin{proof}
Let $\left(\varphi_j\right)_{j\in\N}$  be a countable dense subset of 
$C(\overline\bQ)$ and write $\Psi=\left(\psi_k\right)_{k\in\N}$. 
Then the span of $\varphi_j \psi_k f_i$ is dense in $L^2(\bQ\times\Omega)$ 
(assuming w.l.o.g. that $1\in (f_i)_{i\in\N}$). 
Thus \eqref{eq:ts-conv-conclusion} follows 
from \cite{Hei17}, Lemma  4.4-1, using \eqref{dstar1}, for all 
$\varphi_j \psi_k f_i$. The statement follows eventually from a density 
argument to 
conclude for general $\varphi\in C(\overline \bQ)$.
\end{proof}

\begin{remark}
\label{lem:Extension-Phi-q-stoch-holes}As already observed in \cite{Hei17},
Lemma \ref{lem:Existence-ts-lim} implies that for every 
$f \in L^{\infty}(\Omega)$,
the class of test-functions $\Psi$ can be enriched by a countable
subset $f \Psi\subset L^{2}(\Omega)$ changing $\Omega_{\Psi}$
only by a set of measure $0$.
\end{remark}

We note that the definition of two-scale convergence in \cite{Hei17} 
is formulated in a different way. However, due to Lemma 4.6 of 
\cite{Hei17}, we can recover our
Definition \ref{def:two-scale-conv}.
In particular, the original version of Lemma \ref{lem:Existence-ts-lim}
yields two-scale convergence in the sense of \cite{Hei17}
[Definition 4.2],
and by Lemma 4.6 of \cite{Hei17}  one infers 
Lemma  \ref{lem:Existence-ts-lim}. 
Finally, if $\Omega$ is compact, we recover
the statements of \cite{ZP} by separability of $C_b(\Omega)=C(\Omega)$.

\begin{lemma}
\label{lem:sto-conver-grad}There exists $\tilde{\Omega}\subset\Omega_{\Psi}$
of full measure such that for all $\omega\in\tilde{\Omega}$ the following
holds: If $u^{\eps}\in H^{1}(\bQ;\Rm)$ for all $\eps$, with 
$\|\nabla u^{\eps}\|_{L^{2}(\bQ)}<C$
for $C$ independent from $\eps>0$, then there exists a subsequence
denoted by $\ue$, functions $u\in H^{1}(\bQ;\Rm)$ and 
$v\in L^{2}(\bQ;L_{pot}^{2}(\Omega))$
such that $u^{\eps} \weakto u$ weakly in $H^1(\bQ)$ and 
\[
\nabla u^{\eps}\stackrel{2s}{\weakto}\nabla u+v\qquad\mbox{as }\eps\to0\,.
\]

\end{lemma}
The original version of the above Lemma in \cite{Hei17} was formulated in 
$H^1_0(\bQ)$. However, the proof applies for all sequences in $H^1(\bQ)$. 

We are also interested in the convergence behavior of functions 
$\ue:\,[0,T]\to L^{2}(\bQ)$.
In particular, we provide the following definition:

\begin{definition}
\label{def:weak-A-conv-time}Let $\Psi$ be the set of 
Remark \ref{rem:es-count-dense-set},
$\Lambda=(\varphi_{i})_{i\in\N}$ be a countable dense subset
of $C(\overline{\bQ})$, $\omega\in\Omega_{\Psi}$ and
$\ue\in L^{2}(0,T;L^{2}(\bQ))$
for all $\eps>0$. We say that $(\ue)$ converges (weakly) in two
scales to $u\in L^{2}(0,T;L^{2}(\bQ;L^{2}(\Omega,\P)))$, and
write $\ue\stackrel{2s}{\weakto}u$, if for all continuous and piece-wise
affine functions $\phi:\,[0,T]\to\mathrm{span}\Psi\times\Lambda$
there holds, with 
$\phi_{\omega,\eps}(t,x):=\phi(t,x,\tau_{\frac{x}{\eps}}\omega)$,
\[
\lim_{\eps\to0}\int_{0}^{T}\int_{\bQ}\ue\phi_{\omega,\eps}
dx\, dt=\int_{0}^{T}\int_{\bQ}\int_{\Omega}u(t,x,\tilde \omega) \, 
\phi(t,x,\tilde \omega)\,d\P(\tilde \omega)\,
dx\,dt\,.
\]

\end{definition}
Note that the test functions now have values in the vector space 
$\mathrm{span}\Psi$
since they are affine. Similar to the stationary case, we obtain the
following lemma.

\begin{lemma}
\label{lem:weak-fA-conv-time} (\cite{Hei17}, $Lemma \, 4.16$)  
Let $T>0$. Then, every
sequence $\left(\ue\right)_{\eps>0}$ with 
$\ue\in L^{2}(0,T;L^{2}(\bQ))$
satisfying $\norm{\ue}_{L^{2}(0,T;L^{2}(\bQ))}\leq C$
for some $C>0$ independent from $\eps$ has a weakly two-scale convergent
subsequence with limit function 
$u\in L^{2}(0,T;L^{2}(\bQ;L^{2}(\Omega,\P)))$.
Furthermore, if $\norm{\partial_{t}\ue}_{L^{2}(0,T;L^{2}(\bQ))}
\leq C$
uniformly for $1<p\leq\infty$, then also 
$\partial_{t}u\in L^{2}(0,T;L^{2}(\bQ;L^{2}(\Omega,\P)))$
and $\partial_{t}\ue\stackrel{2s}{\weakto}\partial_{t}u$ in the sense
of Definition \ref{def:weak-A-conv-time} as well as 
$\ue(t)\stackrel{2s}{\weakto}u(t)$
for all $t\in[0,T]$.
\end{lemma}

As a special case of the last result, we have
\begin{lemma}
\label{lem:weak-fA-conv-time-dt-ueps} (\cite{Hei17}, $Lemma \, 4.17$)  
Let $\Psi$ and
$\Omega_{\Psi}$ be given by Remark \ref{rem:es-count-dense-set} and
$\omega\in\Omega_{\Psi}$. Let 
$\ue\in C^{Lip}(0,T;L^{2}(\bQ))$
for all $\eps>0$ such that 
$\norm{\ue}_{C^{Lip}(0,T;L^{2}(\bQ))}\leq C$
for some $C$ independent from $\eps>0$. Then, there exists 
$u\in C^{Lip}(0,T;L^{2}(\bQ;L^{2}(\Omega,\P)))$
and a subsequence $u^{\eps'}$ of $\ue$ such that 
$u^{\eps'}(t)\stackrel{2s}{\weakto}u(t)$
for all $t\in[0,T]$. 
\end{lemma}

\subsection{Domains with holes.}\label{sec:holes}
Since $G(\omega)$ is a random set, there exists, by the considerations
in Section \ref{sec:rand-meas-an-sets}, a set $G\subset\Omega$ such
that $\chi_{G(\omega)}(x)=\chi_{G}(\tau_{x}\omega)$. 
Based on $G$, respectively its complement $G^\complement$,
we obtain the following generalized concept of two-scale convergence.

\begin{lemma}
\label{lem:sto-ts-conv-on-holes} Let $\ue\in L^{2}(\bQ)$ be a sequence
of functions 
such that $\sup_{\eps>0}\norm{\ue}_{L^{2}(\bQ)}<\infty$. 
If $(u^{\eps'})_{\eps'\to0}$ is a subsequence such that 
$u^{\eps'}\stackrel{2s}{\weakto}u$
for some $u\in L^{2}(\bQ;L^{2}(\Omega))$,
then $\ue \, \chi_{\bQ^{\eps}}\stackrel{2s}{\weakto}u \, \chi_{G^\complement}$.
\end{lemma}

\begin{proof}
Let $(u^{\eps'})_{\eps'\to0}$ be a subsequence such that 
$u^{\eps'}\stackrel{2s}{\weakto}u$.
Then the definition of two-scale convergence in $L^{2}(\bQ)$ together
with Remark \ref{lem:Extension-Phi-q-stoch-holes} implies that, for
every $\varphi\in C(\overline{\bQ})$ and $\psi\in\Psi$, it holds
\[
\lim_{\eps\to0} \int_{\bQ}\ue(x) \chi_{G^\complement}(\tau_{\frac{x}{\eps}}
\omega) \varphi(x) \,\psi(\tau_{\frac{x}{\eps}}\omega)\d x=
\int_{\bQ}\int_{\Omega}u(x,\tilde \omega) \chi_{G^\complement}(\tilde \omega)
\varphi(x) \psi(\tilde \omega) \d\P(\tilde \omega)\d x\,.
\]
Furthermore, for $\delta>0$, let us consider the ball $B_{\delta}(x)$
of radius $\delta$ and center $x$.
For $\eps>0$ small enough and
with $\phi_{\omega,\eps}(x):=\varphi(x)\psi(\tau_{\frac{x}{\eps}}\omega)$
it holds that 
\begin{align*}
&\left|\int_{\bQ}\ue(x) \,\left(\chi_{G^\complement}
(\tau_{\frac{x}{\eps}}\omega)-
\chi_{\bQ^{\eps}}(x) \right) \, \phi_{\omega,\eps}(x)\d x\right|  \\
&\leq\left(\int_{B_{\delta}(\partial\bQ)}\left|\ue\right|^{2} \, dx\right)
\left(\int_{B_{\delta}(\partial\bQ)}\varphi^{2}(x)\,\psi^{2}(\tau_{\frac{x}
{\eps}}\omega)\d x\right) \\
&  \leq\left(\int_{B_{\delta}(\partial\bQ)}\varphi^{2}(x)\,
\psi^{2}(\tau_{\frac{x}{\eps}}\omega)\d x\right) \, \sup_{\eps>0}
\norm{\ue}_{L^{2}(\bQ)}\\
 & \to
\left(\int_{B_{\delta}(\partial\bQ)}\int_{\Omega}\varphi^{2}(x)\,
\psi^{2}(\omega)\d\P(\omega)\d x\right)\sup_{\eps>0}\norm{\ue}_{L^{2}(\bQ)}\,.
\end{align*}
Since $\delta>0$ is arbitrary, the statement follows.
\end{proof}

\begin{lemma} \label{lem}
Let $\ue\in H^{1}(\bQ^{\eps}(\omega))$ be a sequence of functions
such that $\sup_{\eps>0}\norm{\ue}_{H^{1}(\bQ^{\eps}(\omega))}<\infty$.
Then there exist functions $u\in H^{1}(\bQ)$ and $v \in L^{2}(\bQ;L_{\pot}^{2}
(\Omega))$
such that $\sE_{\eps}\ue\weakto u$ weakly in $H^{1}(\bQ)$ as well
as $\ue\stackrel{2s}{\weakto}\chi_{G^\complement} \,u$ and 
$\nabla\ue\stackrel{2s}{\weakto}\chi_{G^\complement}\nabla u+
\chi_{G^\complement} \, v$.
\end{lemma}
\begin{proof}
Lemma \ref{lem:extension-Op} implies that 
$\sup_{\eps>0}\norm{\sE_{\eps}\ue}_{H^{1}(\bQ)}<\infty$.
Hence, due to Lemma \ref{lem:sto-conver-grad} there exists $u\in H^{1}(\bQ)$
and $v \in L^{2}(\bQ;L_{\pot}^{2}(\Omega))$ such that $\sE_{\eps}\ue\weakto u$
weakly in $H^{1}(\bQ)$ and $\nabla (\sE_{\eps}\ue) \stackrel{2s}{\weakto}
\nabla u+v$.
Lemma \ref{lem:sto-ts-conv-on-holes} now implies 
$\ue\stackrel{2s}{\weakto}\chi_{G^\complement} \,u$
and $\nabla\ue\stackrel{2s}{\weakto}\chi_{G^\complement}\nabla u+
\chi_{G^\complement} \,v$. 
\end{proof}

\begin{lemma} \label{lem:conv-h1-time}
Let $\ue\in L^2(0,T;H^{1}(\bQ^{\eps}(\omega)))$ be a sequence of functions
such that 
$$\sup_{\eps>0}\norm{\ue}_{L^2(0,T;H^{1}(\bQ^{\eps}(\omega)))}+
\norm{\partial_t\ue}_{L^2(0,T;L^{2}(\bQ^{\eps}(\omega)))}<\infty\,.$$
Then there exist functions $u\in L^2(0,T;H^{1}(\bQ))$ with 
$\partial_tu\in L^2(0,T;L^2(\bQ))$ and 
$v \in L^2(0,T;L^{2}(\bQ;L_{\pot}^{2}(\Omega)))$
such that $\sE_{\eps}\ue\weakto u$ weakly in $L^2(0,T;H^{1}(\bQ))$ and 
$\sE_{\eps}\ue\to u$ strongly in $L^2(0,T;L^{2}(\bQ))$ as well 
as 
$$\ue\stackrel{2s}{\weakto}\chi_{G^\complement} \,u\,,\quad 
\partial_t\ue\stackrel{2s}{\weakto}\chi_{G^\complement} \,\partial_tu\,,
\quad \text{and}\quad 
\nabla\ue\stackrel{2s}{\weakto}\chi_{G^\complement}\nabla u+
\chi_{G^\complement} \, v\,.$$
\end{lemma}
\begin{proof}
We only have to prove $\sE_{\eps}\ue\to u$ strongly in $L^2(0,T;L^{2}(\bQ))$ 
since the remaining part of the statement has either been demonstrated above 
or can be obtained by generalizing previous considerations.

We first observe that, for every times $t_1,t_2\in[0,T]$, it holds by 
Lemma \ref{lem:extension-Op} that
\begin{align*}\norm{\int_{t_1}^{t_2}\sE_{\eps}\ue(t)\d t }_{H^1(\bQ)}
& \leq \norm{\sE_{\eps}\int_{t_1}^{t_2}\ue(t)\d t }_{H^1(\bQ)} \\
& \leq C\norm{\int_{t_1}^{t_2}\ue(t)\d t }_{H^1(\bQ^\eps)}\\
& \leq CT^{\frac12}\norm{\ue}_{L^2(0,T;H^1(\bQ^\eps))}
\end{align*}
and hence $\left\{\int_{t_1}^{t_2}\sE_{\eps}\ue(t)\d t\right\}_{\eps>0}$ is 
precompact in $L^2(\bQ)$. Next, one can write by using again 
Lemma \ref{lem:extension-Op}:
\begin{align*}
    \int_0^{T-h}\left\Vert \sE_{\eps}(\ue(t)-\ue(t+h))\right\Vert^2_{L^2(\bQ)}
dt 
    & \leq C\int_0^{T-h}\left\Vert \ue(t)-\ue(t+h) \right
\Vert^2_{L^2(\bQ^\eps)} \, dt\\
    & \leq C\int_0^{T-h}\left\Vert \int_t^{t+h}\partial_t\ue(s)\d s 
\right\Vert^2_{L^2(\bQ^\eps)} \, dt\\
    & \leq C\int_0^{T-h}h\left\Vert \partial_t\ue \right
\Vert^2_{L^2(t,t+h;L^2(\bQ^\eps))} \, dt \\
    & \leq C h\norm{\partial_t\ue}^2_{L^2(0,T;L^2(\bQ^\eps))}
\end{align*}
where the constant $C$ changes in the last step. 
Since it holds $\sE_{\eps} \ue \weakto u$ in $L^2(0,T;L^{2}(\bQ))$, 
we conclude from Simon's compactness theorem 
(see Theorem 1 of \cite{simon1986}).
\end{proof}

\vskip2mm

\section{Setting of the problem and estimates} \label{sett:prob}

Throughout this paper, $\varepsilon$ will denote the general term of a
sequence of positive reals which converges to zero.
We consider in the following a system of anisotropic diffusion-coagulation
Smoluchowski-type equations which
describes the dynamics of cluster growth.
In particular, we introduce the vector-valued random function
$u^{\varepsilon}: [0,T] \times \bQ^\eps \rightarrow 
\mathbb{R}^M,$
$u^{\varepsilon}=(u_1^{\varepsilon}, \ldots, u_M^{\varepsilon})$ (with $M \in 
\mathbb{N}$ being fixed) where the variable $u_s^{\varepsilon} \geq 0$
($1 \leq s < M$) represents the concentration of $s$-clusters, that is,
clusters consisting of $s$ identical elementary particles (monomers), while
$u_M^{\varepsilon} \geq 0$ takes into account aggregations of more than
$M-1$ monomers.
We assume that the only reaction allowing clusters to coalesce to form
larger clusters is a binary coagulation mechanism,
while the movement of clusters results only from a diffusion process
described by 
a stationary ergodic random matrix 
$$
\Big(d^s_{i,j}(t, x, \tau_{\frac{x}{\varepsilon}} 
\omega)\Big)_{i,j=1,\dots,m} =: 
D_s (t, x, \tau_{\frac{x}{\varepsilon}} \omega)
\qquad 
1 \leq s \leq M,
$$
where $(t,x)\in [0,T] \times \bQ$.
Here  $D_s(t, x, \tau_{\frac{x}{\varepsilon}}\omega)$ is the realization 
(see Remark \ref{r2.1}) of
a  random matrix. 
Indeed, aging (as well as the AD itself) yields atrophy of the cerebral 
parenchyma,
inducing changes in the diffusion rate of the amyloid agglomerates. 
In addition, this
rate may vary for different regions of the brain. Finally, we have to take 
into account that A$\beta$ aggregates do not diffuse
freely in an uniform fluid: the cerebral tissue 
 consists of  large non-neuronal support cells  (the macroglia)
and the  A$\beta$ polymers move within the cerebrospinal fluid  
along the interstices between these cells
that, in turn, are stochastically distributed.

With these notations, our system reads:

\begin{eqnarray} \label{2b.1}
\begin{cases}
\frac{\displaystyle \partial{u_1^{\varepsilon}}}{\displaystyle \partial t}-
div (D_1 (t, x, \tau_{\frac{x}{\varepsilon}} \omega) \, 
 \nabla_x u_1^{\varepsilon})+u_1^{\varepsilon} \, \sum_{j=1}^M a_{1,j}
u_j^{\varepsilon}=0 & \text{in } [0,T] \times \bQ^\eps \\

\\ 
[D_1 (t, x, \tau_{\frac{x}{\varepsilon}} \omega) \, 
\nabla_x u_1^{\varepsilon}]
\cdot n=0 & \text{on } [0,T] \times \partial\bQ  \\

\\
[D_1 (t, x, \tau_{\frac{x}{\varepsilon}} \omega) \, 
\nabla_x u_1^{\varepsilon}] 
\cdot \nu_{\Gamma_{\bQ}^\eps}=
\varepsilon \, \eta(t, x, \tau_{\frac{x}{\varepsilon}} \omega) & \text{on } 
[0,T] \times \Gamma_{\bQ}^\eps \\

\\
u_1^{\varepsilon}(0,x)=U_1 & \text{in } \bQ^\eps
\end{cases}
\end{eqnarray}

if $1 < s <M$

\begin{eqnarray} \label{2b.2}
\begin{cases}
\frac{\displaystyle \partial{u_s^{\varepsilon}}}{\displaystyle \partial t}-
div (D_s (t, x, \tau_{\frac{x}{\varepsilon}} \omega)
\, \nabla_x u_s^{\varepsilon})+u_s^{\varepsilon} \, \sum_{j=1}^M a_{s,j}
u_j^{\varepsilon}=f^{\varepsilon} & \text{in } [0,T] \times \bQ^\eps \\

\\ 
[D_s (t, x, \tau_{\frac{x}{\varepsilon}} \omega) \, 
\nabla_x u_s^{\varepsilon}]
\cdot n=0 & \text{on } [0,T] \times \partial\bQ  \\

\\
[D_s (t, x, \tau_{\frac{x}{\varepsilon}} \omega) \,
\nabla_x u_s^{\varepsilon}] \cdot \nu_{\Gamma_{\bQ}^\eps}=0
& \text{on } 
[0,T] \times \Gamma_{\bQ}^\eps \\

\\
u_s^{\varepsilon}(0,x)=0  & \text{in } \bQ^\eps
\end{cases}
\end{eqnarray}

and eventually

\begin{eqnarray} \label{2b.3}
\begin{cases}
\frac{\displaystyle \partial{u_M^{\varepsilon}}}{\displaystyle \partial t}-
div (D_M (t, x, \tau_{\frac{x}{\varepsilon}} \omega)
\, \nabla_x u_M^{\varepsilon})
=g^{\varepsilon} & \quad \text{in } [0,T] \times \bQ^\eps \\

\\ 
[D_M (t, x, \tau_{\frac{x}{\varepsilon}} \omega) \, 
\nabla_x u_M^{\varepsilon}]
\cdot n=0 & \quad \text{on } [0,T] \times \partial\bQ \\

\\
[D_M (t, x, \tau_{\frac{x}{\varepsilon}} \omega) \, 
\nabla_x u_M^{\varepsilon}] 
\cdot \nu_{\Gamma_{\bQ}^\eps}=0
& \quad \text{on } 
[0,T] \times \Gamma_{\bQ}^\eps \\

\\
u_M^{\varepsilon}(0,x)=0  & \quad \text{in } \bQ^\eps
\end{cases}
\end{eqnarray}
where the
gain terms $f^{\varepsilon}$ and $g^{\varepsilon}$ in (\ref{2b.2})
and (\ref{2b.3}) are given by

\begin{equation} \label{2b.4}
f^{\varepsilon}=\frac{1}{2} \, \sum_{j=1}^{s-1} a_{j,s-j} \, 
u_j^{\varepsilon} \,
u_{s-j}^{\varepsilon}
\end{equation}

\begin{equation} \label{2b.5}
g^{\varepsilon}=\frac{1}{2} \, \sum_{\substack{j+k \geq M \\ 
k< M (\text{if } j=M) \\ j<M (\text{if } k=M)}} 
a_{j,k} \, u_j^{\varepsilon} \, u_k^{\varepsilon}.
\end{equation}
The kinetic coefficients $a_{i,j}$ represent a reaction in which an
($i+j$)-cluster is formed from an $i$-cluster and a $j$-cluster.
Therefore, they can be interpreted as "coagulation rates" and are symmetric
$a_{i,j}=a_{j,i} >0$  ($i,j=1, \ldots, M$), but $a_{M,M}=0$.
Let us remark that the meaning of $u_M^{\varepsilon}$ differs from that of
$u_s^{\varepsilon}$ ($s < M$), since it describes the sum of the 
densities of all
the 'large' assemblies. It is assumed that large assemblies exhibit all the
same coagulation properties and do not coagulate with each other.

The production of $\beta$-amyloid peptide by the malfunctioning neurons is
described imposing a non-homogeneous Neumann condition on the boundary
of the holes, randomly selected within our domain.
To this end, we consider on $\Gamma_{\bQ}^\eps$ in Eq. (\ref{2b.1}) a
stationary ergodic random function 
$\eta=\eta(t,x, \tau_{\frac{x}{\varepsilon}}\omega)$.
Here  $\eta(t,x, \tau_{\frac{x}{\varepsilon}}\omega)$ is the realization 
(see Remark \ref{r2.1}) of
a  random function: 

\begin{equation} \label{2b.6}
\eta : [0,T] \times \overline{\bQ}\times \Omega \rightarrow [0,1]
\end{equation}
where the value '0' is assigned to 'healthy' neurons 
while all the other
values in $]0,1]$ indicate different degrees of malfunctioning.
Moreover, we assume that $\eta$ is an increasing function of time, since once
the neuron has become 'ill', it can no longer regain its original state of
health.

\bigskip

Further hypotheses are listed below:
\\
(H.1) the diffusion coefficients satisfy
$d^s_{i,j}\in C^1\left([0,T]\times\bQ;C_b^1(\Omega)\right)$
for $i,j=1,\dots,m$,
 $s=1,\dots, M$.
  We put $$\Lambda^\star:= 
\sup_{i,j, s} \|d_{i,j}^s\|_{C^1\left([0,T]\times\bQ;C_b^1(\Omega)\right)}.
$$
\\
In particular, the map
$
(t,x,\omega)\to D_s(t,x,\tau_{\frac{x}{\varepsilon}}\omega)$ is 
continuously differentiable;
\\
(H.2) $d^s_{i,j} = d^s_{j,i}$, for $i,j=1,\dots, m$, $s=1,\dots, M$;
\\
(H.3)   there exists $0<\lambda\le \Lambda$ such that
$$
\lambda |\xi|^2\le\sum_{i,j=1}^m d^s_{i,j}(t,x,\tau_{\frac{x}{\varepsilon}} 
\omega)\xi_i\xi_j \le 
\Lambda |\xi|^2
$$
for all $s=1,\dots, M$, $\xi\in \mathbb R^m$, $(t,x)\in [0,T]\times 
\overline{\bQ}$
and for $\mathbb P$-a.e. $\omega\in \Omega$.

\bigskip
Moreover, the function $\eta$, appearing in (\ref{2b.1}), is a given bounded 
function
satisfying the following conditions:
\\
(H.4) $\eta\in C^1\left([0,T]\times\bQ;C_b^1(\Omega)\right)$; 
\\
(H.5) $\eta(0, \cdot,\cdot)=0$ and
$U_1$ is a positive constant such that
\begin{equation}\label{star}
U_1 \leq \|\eta\|_{L^\infty([0,T]\times \overline{\bQ}\times\Omega)}.
\end{equation}   
We can repeat now almost verbatim the arguments of \cite{FL_wheeden}, 
Theorems 2.1, 2.2, 2.3
and 2.4 to obtain the following ``deterministic'' 
(i.e. for fixed $\omega\in \Omega$) existence and
regularity result.

\begin{theorem}\label{deterministic} 
Suppose Assumption \ref{1} (where additionally $G(\omega)$ has a smooth 
boundary) and  
(H.1) - (H.5) hold. Then 
for $\mathbb P$-a.e. $\omega\in\Omega$ and for any $\varepsilon>0$ the system 
\eqref{2b.1} - \eqref{2b.3}
admits a unique maximal classical solution 
$$
u^\varepsilon_\omega=
(u_{\omega,1}^{\varepsilon},\dots , u_{\omega,M}^{\varepsilon})
$$
such that

\begin{itemize}
\item[(i)] there exists $\alpha\in (0,1)$, $\alpha$ 
depending only on $N, \lambda,
\Lambda^\star$,  $\varepsilon$ and $\omega$, such that 
$u^\varepsilon\in C^{1+\alpha/2, 2+\alpha}
([0,T]\times\bQ^\varepsilon,\mathbb R^M)$ for 
$\mathbb P$-a.e. $\omega\in \Omega$
and
\begin{equation}\label{april 8}
\| u^\varepsilon_\omega\|_{C^{1+\alpha/2, 2+\alpha}([0,T]\times
\bQ^\varepsilon,
\mathbb R^M)}
\le C_0 = C_0(U_1, \Vert \eta \Vert_{L^\infty([0,T]\times 
\overline{\bQ}\times\Omega)}, K,\varepsilon, \omega,\alpha);
\end{equation}
\item[(ii)] 
$u_{\omega,j}^\varepsilon(t,x)>0$  for 
$(t,x)\in [0,T]\times \bQ^\varepsilon$, $\mathbb P$-a.e. $\omega\in\Omega$ 
and $j=1,\dots, M$.
\end{itemize}
\end{theorem}

In the sequel we shall rely on the fact that statements that hold 
$\mathbb P$-a.e.
can be seen as deterministic assertions, since they hold whenever 
$\bQ^\varepsilon$
is a set enjoying the regularity properties described in Remark \ref{7 feb:1}, 
Assumption \ref{1} and Remark \ref{7 feb:2}.

Arguing as in \cite{FL_wheeden}, the first and crucial step will consist of 
proving that the $u^\varepsilon_{\omega,j}$ are
equibounded in $L^\infty([0,T] \times \bQ^{\varepsilon})$ for 
$\mathbb P$-a.e. $\omega\in\Omega$ and
$j=1,\dots,M$.

In particular, an uniform bound for $u_{\omega}^{\varepsilon}$ in 
$L^{\infty} ([0, T] \times 
\bQ^{\varepsilon})$ is provided by the following statement: 

\begin{theorem} \label{l4.2}
Let $u^\varepsilon_\omega=(u_{\omega,1}^{\varepsilon},\dots , 
u_{\omega,M}^{\varepsilon})$ be 
as in Theorem \ref{deterministic}.
Then

 \begin{equation} \label{4.15} 
\Vert u_{\omega,1}^{\varepsilon} \Vert_{L^{\infty}
([0,T] \times \bQ^{\varepsilon})}
\leq 
|U_1|+
c \, \Vert \eta \Vert_{L^\infty([0,T]\times 
\overline{\bQ}\times\Omega)},
\end{equation}
for $\mathbb P$-a.e. $\omega\in\Omega$,
where $c$ is independent of $\varepsilon>0$.

In addition, there exists $K>0$ such that

\begin{equation} \label{4.61}
\Vert u_{\omega,j}^{\varepsilon} \Vert_{L^{\infty} 
([0,T] \times \bQ^{\varepsilon})}
\leq K
\end{equation}
for $\mathbb P$-a.e. $\omega\in\Omega$, uniformly with respect to 
$\varepsilon>0$.
\end{theorem}

\begin{proof} Thanks to extension Lemma \ref{lem:extension-Op}, 
the function $u^\varepsilon_\omega$
can be continued on all $[0,T] \times \bQ$. Therefore we can repeat step by 
step the arguments of \cite{FL_wheeden},
Theorems 2.2 and 2.3, that in turn rely on \cite{ladyzenskaya_et_al} 
(see also \cite{nittka} and \cite{wrz}).
\end{proof}

Therefore
 
\begin{theorem}[\cite{FL_wheeden}, Theorems 3.1. and 3.2] \label{t1.6}
The sequence $(\nabla_x u_{\omega,j}^{\varepsilon})_{\varepsilon >0}$ 
($1 \leq j \leq M$) is bounded in
$L^2([0,T] \times \bQ^{\varepsilon})$ for $\mathbb P$-a.e. $\omega\in \Omega$, 
uniformly in $\varepsilon$.

In addition, the sequence 
$(\partial_t u_{\omega,j}^{\varepsilon})_{\varepsilon >0}$  ($1 \leq j \leq M$)
is bounded in
$L^2([0,T] \times \bQ^{\varepsilon})$ for $\mathbb P$-a.e. $\omega\in \Omega$, 
uniformly in $\varepsilon$.
\end{theorem}

\vskip2mm

\section{Homogenization } \label{sec:homo}

Our main statement shows that it is possible to homogenize the set of 
Eqs. (\ref{2b.1})-(\ref{2b.3}) as $\varepsilon \rightarrow 0$.

\begin{theorem} \label{homo_twoscale}
Let $u_s^{\varepsilon} (t,x)$ ($1 \leq s \leq M$) be a family of nonnegative
classical solutions to the system (\ref{2b.1})-(\ref{2b.3}).
Denote by a tilde the extension by zero outside $\bQ^\eps (\omega)$ and let
$\chi_{G^\complement}$ represent the characteristic function of the random
set ${G^\complement}(\omega)$.
Then, the sequences $(\widetilde{u_s^{\varepsilon}})_{\varepsilon>0}$,
$(\widetilde{\nabla_x u_s^{\varepsilon}})_{\varepsilon>0}$ and
$(\widetilde{\partial_t u_s^{\varepsilon}})_{\varepsilon>0}$ 
($1 \leq s \leq M$)
stochastically two-scale converge to:
$[\chi_{G^\complement} \, u_s(t,x)]$,
$[\chi_{G^\complement} (\nabla_x u_s(t,x)+v_s(t,x,\omega))]$,
$[\chi_{G^\complement} \, \partial_t \, u_s(t,x)]$ ($1 \leq s \leq M$),
respectively.
The limiting functions 
$[ (t,x) \mapsto u_s(t,x), (t,x,\omega) \mapsto v_s(t,x,\omega) ]$
($1 \leq s \leq M$) are the unique solutions lying in
$L^2 (0,T; H^1 (\bQ)) \times L^2 ([0,T] \times \bQ; L_{\pot}^{2}(\Omega)) $
of the following two-scale homogenized systems:

If $s=1$:

\begin{eqnarray} \label{5.15} 
\begin{cases}
\theta \, \frac{\displaystyle \partial u_1}{\displaystyle \partial t}(t,x)
-div_x \bigg[  D_1^{\star}(t,x) \, 
\nabla_x u_1(t,x) \bigg] \\
+\theta \, u_1(t,x)
\sum_{j=1}^M a_{1,j} \, u_j(t,x)  
= \displaystyle \int_{\Omega} \chi_{\Gamma_{G^\complement}}\,
\eta(t,x,\omega) \, d\mugammapalm(\omega)  
& \text{ in }  [0,T] \times \bQ\\

[ D_1^{\star}(t,x) \, \nabla_x u_1 (t,x)] \cdot n=0 & \text{ on }  
[0,T] \times \partial \bQ \\

u_1(0,x)=U_1 & \text{ in }  \bQ
\end{cases}
\end{eqnarray}

If $1 < s < M$:

\begin{eqnarray} \label{5.16} 
\begin{cases}
\theta \, \frac{\displaystyle \partial u_s}{\displaystyle \partial t}(t,x)
-div_x \bigg[  D_s^{\star}(t,x) \, 
\nabla_x u_s(t,x) \bigg] \\
+\theta \,  u_s(t,x)
\sum_{j=1}^M a_{s,j} \,  u_j(t,x) \\ 
=\frac{\displaystyle \theta}{\displaystyle 2} \sum_{j=1}^{s-1} a_{j,{s-j}}
\, u_j(t,x) \, u_{s-j}(t,x)
& \text{ in }  [0,T] \times \bQ  \\

\\
[ D_s^{\star}(t,x) \, \nabla_x u_s (t,x)] \cdot n=0 & \text{ on }  
[0,T] \times \partial \bQ \\

u_s(0,x)=0 & \text{ in }  \bQ
\end{cases}
\end{eqnarray}

If $s=M$:

\begin{eqnarray} \label{5.17} 
\begin{cases}
\theta \, \frac{\displaystyle \partial u_M}{\displaystyle \partial t}(t,x)
-div_x \bigg[  D_M^{\star}(t,x) \, 
\nabla_x u_M(t,x) \bigg]  \\
=\frac{\displaystyle \theta}{\displaystyle 2} 
\sum_{\substack{j+k \geq M \\ k< M (\text{if } j=M) \\ 
j<M (\text{if } k=M)}} a_{j,k} \,
u_j(t,x) \, u_k(t,x)
& \text{ in }  [0,T] \times \bQ \\

\\
[ D_M^{\star}(t,x) \, \nabla_x u_M (t,x)] \cdot n=0 & \text{ on }  
[0,T] \times \partial \bQ \\

u_M(0,x)=0 & \text{ in }  \bQ
\end{cases}
\end{eqnarray}
where
$$\theta=\int_{\Omega} \chi_{G^\complement} \, d\mupalm({\omega})=
\P(G^\complement)$$
represents the fraction of volume occupied by ${G^\complement}$ and, 
for every $1\leq s\leq M$, 
$D_s^{\star}(t,x)$ is a deterministic matrix, called "effective diffusivity", 
defined by

$$(D_s^{\star})_{ij}(t,x)=\displaystyle \int_{\Omega} \chi_{G^\complement} \,
D_s(t,x,\omega) ( w_i(t,x,\omega)+ \hat{e}_i) \cdot
(w_j(t,x,\omega)+ \hat{e}_j) \,d\P({\omega}) $$
with $\hat{e}_i$ being the $i$-th canonical unit vector in $\mathbb{R}^m$, and
$(w_i)_{1 \leq i \leq m} \in L^2 ([0,T] \times \bQ; 
L_{\pot}^{2}(G^\complement))$ the family of 
solutions of the following microscopic problem
\begin{eqnarray} \label{5.18}
\begin{cases}
-div_{\omega} [D_s(t,x,\omega)  (w_i(t,x,\omega)+ \hat{e}_i)]=0 \, \, \, \qquad 
\; \; \; \; \; 
\text{ in} \, \, G^\complement \\
D_s(t,x,\omega) [w_i(t,x,\omega)+\hat{e}_i] \cdot \nu_{\Gamma_{G^\complement}}
=0 \, \, \, \qquad 
\; \; 
\text{ on} \, \, \Gamma_{G^\complement}.
\end{cases}
\end{eqnarray}
Finally,
$$v_s(t,x,\omega)=\sum_{i=1}^m\, w_i (t,x,\omega) \,
\frac{\displaystyle \partial u_s}{\displaystyle \partial x_i}(t,x)
\; \; \; (1 \leq s \leq M).$$

\end{theorem}

\begin{proof}
In view of Theorems \ref{l4.2} and \ref{t1.6}, the sequences 
$\widetilde{(u_s^{\varepsilon})}_{\varepsilon > 0}$,
$\widetilde{(\nabla_x u_s^{\varepsilon})}_{\varepsilon > 0}$ and
$ \widetilde{\bigg(\frac{\displaystyle \partial u_s^{\varepsilon}}
{\displaystyle \partial t}\bigg)}_{\varepsilon > 0}$ ($1 \leq s \leq M$) are
bounded in $L^2 ([0,T] \times \bQ)$.
Using Lemma \ref{lem:conv-h1-time}, they two-scale converge,
up to a subsequence, respectively, to:
$[\chi_{G^\complement} \, u_s(t,x)]$, 
$[\chi_{G^\complement} (\nabla_x u_s(t,x)+v_s (t,x,\omega))]$,
$[\chi_{G^\complement} \partial_t u_s(t,x)]$,
 where $u_s \in L^2 (0,T; H^1(\bQ))$ and
$v_s \in L^2 ([0,T] \times \bQ; L_{\pot}^{2}(\Omega))$.
As test functions for homogenization, let us take
\begin{equation} \label{5.1}
\phi^{\varepsilon}(t,x,\omega):=\phi_0(t,x)+\varepsilon \, \phi(t,x) 
\psi(\tau_{\frac{x}{\varepsilon}}\omega)
\end{equation}
where $\phi_0, \phi \in C^1([0,T] \times \overline{\bQ})$ and $\psi\in\Psi$,
with $\Psi$ being the set of Remark \ref{rem:es-count-dense-set}.

In the case when $s=1$, let us multiply the first equation of (\ref{2b.1}) by
the test function $\phi^{\varepsilon}$.
Integrating, the divergence theorem yields
\begin{equation} \label{5.2} 
\begin{split}
&\displaystyle \int_0^T \int_{\bQ^\eps(\omega)} 
\frac{\displaystyle \partial u_1^{\varepsilon}}{\displaystyle \partial t} \,
\phi^{\varepsilon} (t,x,\omega) \, dx \, dt+
 \displaystyle \int_0^T \int_{\bQ^\eps(\omega)} 
\bigg< D_1(t,x,\tau_{\frac{x}{\varepsilon}}\omega) \nabla_x u_1^{\varepsilon}, 
\nabla \phi^{\varepsilon}  \bigg> \, dx \, dt \\ 
&+\displaystyle \int_0^T \int_{\bQ^\eps(\omega)} u_1^{\varepsilon}
\sum_{j=1}^M a_{1,j} \, u_j^{\varepsilon} \, \phi^{\varepsilon}(t,x,\omega) 
\, dx \, dt=
\varepsilon \, \displaystyle \int_0^T \int_{\Gamma_{\bQ}^\eps(\omega)}
\eta ( t,x,\tau_{\frac{x}{\varepsilon}}\omega ) \, 
\phi^{\varepsilon}(t,x,\omega) 
\, d\mathcal{H}^{m-1} \, dt.
\end{split}
\end{equation}
Passing to the two-scale limit, as $\eps \rightarrow 0$, we get, 
taking into account \eqref{rescaling 2}:
\begin{eqnarray} \label{5.3} 
&\displaystyle \int_0^T \int_{\bQ} \int_{\Omega} \chi_{G^\complement}
\frac{\displaystyle \partial u_1}{\displaystyle \partial t}(t,x) \, 
\phi_0(t,x) \, d\P({\omega}) \, dx \, dt \nonumber \\
&+ \displaystyle \int_0^T \int_{\bQ} \int_{\Omega} \chi_{G^\complement}
D_1(t,x,\omega)
[\nabla_x u_1(t,x)+v_1(t,x,\omega)] \nonumber \\ 
&\cdot
[\nabla_x \phi_0(t,x)+\phi(t,x) \nabla_{\omega} \psi(\omega)] \, 
d\P({\omega}) \, dx \, dt \nonumber \\ 
&+\displaystyle \int_0^T \int_{\bQ} \int_{\Omega} \chi_{G^\complement} 
u_1(t,x)
\sum_{j=1}^M a_{1,j} \, u_j(t,x) \, \phi_0(t,x) \, 
d\P({\omega}) \, dx \, dt \nonumber \\
&= \displaystyle \int_0^T \int_{\bQ} \int_{\Omega} 
\chi_{\Gamma_{G^\complement}} \,
\eta(t,x,\omega) \,
\phi_0(t,x) \, d\mugammapalm(\omega)  \, dx \, dt.
\end{eqnarray}
The term on the right-hand side follows from
Eq. (\ref{dstar2}).
The last term on the left-hand side of (\ref{5.3}) has been obtained by
observing that $\sE_\eps u_j^\eps\to u_j$ strongly in $L^2(0,T;L^2(\bQ))$ 
(see Lemma \ref{lem:conv-h1-time}) 
and that the two-scale convergence of 
$u_1^\eps\stackrel{2s}{\weakto}\chi_{G^\complement}u_1$ 
implies weak convergence of
$u_1^\eps \phi^\eps(\cdot,\cdot,\omega)\weakto u_1 \phi_0 
\int_\Omega \chi_{G^\complement} d\P({\omega})$ 
in $L^2(0,T;L^2(\bQ))$.

An integration by parts shows that (\ref{5.3}) can be put in the strong
form associated with the following homogenized system:

\begin{equation} \label{5.4}
-div_{\omega} [D_1(t,x,\omega) (\nabla_x u_1(t,x)+v_1(t,x,\omega))]=0 \, \, \, 
\qquad
\; \; \; \text{ in} \, \,
[0,T] \times {\bQ} \times {G^\complement}
\end{equation}

\begin{equation} \label{5.5}
[D_1(t,x,\omega) (\nabla_x u_1(t,x)+v_1(t,x,\omega))] \cdot 
\nu_{\Gamma_{G^\complement}} 
=0  \, \, \, 
\qquad 
\; \; \; \; \; \; \; \; \;  \; \; \text{ on} \, \, 
[0,T] \times {\bQ} \times \Gamma_{G^\complement}
\end{equation}

\begin{equation} \label{5.6} 
\begin{split}
&\theta \, \frac{\displaystyle \partial u_1}{\displaystyle \partial t}(t,x)
-div_x \bigg[  \displaystyle \int_{\Omega} \chi_{G^\complement} \, 
D_1(t,x,\omega)
(\nabla_x u_1(t,x)+v_1(t,x,\omega)) d\P({\omega}) \bigg] \\
&+\theta \, u_1(t,x)
\sum_{j=1}^M a_{1,j} \, u_j(t,x) 
- \displaystyle \int_{\Omega}  \chi_{\Gamma_{G^\complement}} \,
\eta(t,x,\omega) \, d\mugammapalm(\omega)=0 \, \, \, 
\; \; \; \text{ in} \, \, [0,T] \times {\bQ}
\end{split}
\end{equation}

\begin{equation} \label{5.7}
\bigg[ \displaystyle \int_{\Omega} \chi_{G^\complement} \,
D_1(t,x,\omega) (\nabla_x u_1(t,x)+
v_1(t,x,\omega)) 
\, d\P({\omega}) \bigg] \cdot n =0  \, \, \, \qquad 
\text{ on} \, \, 
[0,T] \times \partial\bQ
\end{equation}
where
\begin{equation} \label{5.8}
\theta=\int_{\Omega} \chi_{G^\complement} \, d\P({\omega})=\P(G^\complement)
\end{equation}
represents the fraction of volume occupied by $G^\complement$.
To conclude, by continuity, we have that
$$u_1(0,x)=U_1 \, \, \, \qquad \; \; \; \; \; \text{ in} \, \, \bQ.$$
The function $v_1(t,x,\omega)$, satisfying (\ref{5.4}) and (\ref{5.5}),
can be expressed as follows
\begin{equation} \label{5.9}
v_1(t,x,\omega):=\sum_{i=1}^m \, w_i (t,x,\omega) \, 
\frac{\displaystyle \partial u_1}{\displaystyle \partial x_i}(t,x)
\end{equation}
where $(w_i)_{1 \leq i \leq m} \in L^2 ([0,T] \times \bQ;
L_{\pot}^{2}(G^\complement))$
is the family of solutions of the microscopic problem
\begin{eqnarray} \label{5.10}
\begin{cases}
-div_{\omega} [D_1(t,x,\omega) (w_i(t,x,\omega)+ \hat{e}_i)]=0 \, \, \, \qquad 
\text{ in} \, \, G^\complement \\
D_1(t,x,\omega) [w_i(t,x,\omega)+\hat{e}_i] \cdot \nu_{\Gamma_{G^\complement}}
=0 
\, \, \, \qquad 
\; \;
\text{ on} \, \, \Gamma_{G^\complement} 
\end{cases}
\end{eqnarray}
and $\hat{e}_i$ is the $i$-th unit vector of the canonical basis of 
${\mathbb R}^m$.
The system (\ref{5.10}) represents the stochastic version of the
"cell problem" defined in periodic homogenization.
By using the relation (\ref{5.9}) in Eqs.(\ref{5.6}) and (\ref{5.7}), we get
\begin{equation} \label{5.11} 
\begin{split}
\theta \, \frac{\displaystyle \partial u_1}{\displaystyle \partial t}(t,x)
&-div_x \bigg[  D_1^{\star}(t,x) \, 
\nabla_x u_1(t,x) \bigg]
+\theta \, u_1(t,x)
\sum_{j=1}^M a_{1,j} \, u_j(t,x)  \\
&- \displaystyle \int_{\Omega} \chi_{\Gamma_{G^\complement}} \, 
\eta(t,x,\omega) \, d\mugammapalm(\omega)=0 \, \, \, 
\; \; \; \text{ in} \, \, [0,T] \times \bQ
\end{split}
\end{equation}
\begin{equation} \label{5.12}
[ D_1^{\star}  \nabla_x u_1 (t,x)] \cdot n=0 
\; \; \; \; \; \;  \; \; \; \; \; \; \; \;
\text{ on} \, \, 
[0,T] \times \partial\bQ 
\end{equation}
where the entries of the matrix $D_1^{\star}$ (called "effective diffusivity")
are given by
\begin{equation} \label{5.13}
(D_1^{\star})_{ij}(t,x)=\displaystyle \int_{\Omega} \chi_{G^\complement} \,
D_1(t,x,\omega)
[ w_i(t,x,\omega)+ \hat{e}_i] \cdot
[w_j(t,x,\omega)+ \hat{e}_j] \, d\P({\omega}).
\end{equation}
The proof for the case $1< s \leq M$ is achieved by applying exactly the
same arguments.

\end{proof}

\vskip2mm

\appendix 
\section{Appendix A} \label{appA}

We review some basic results on the realization of random domains based on
continuum percolation theory \cite{MR}.

\subsection{Stationary ergodic point processes.} \label{appA.1} 
Since in percolation theory, random modeling is based on the occurrences
of stationary point processes, in this section, we state their definition
and some basic properties \cite{DV}.

\begin{definition} \label{dA.1}
Denote the $\sigma$-algebra of Borel sets in $\mathbb{R}^m$ by
$\mathcal{B}^m$.

(i) A Borel measure $\mu$ on $\mathbb{R}^m$ is boundedly finite if
$\mu (A) < \infty$ for every bounded Borel set $A$.

(ii) Let $N$ be the space of all boundedly finite integer-valued
measures on $\mathcal{B}^m$, called counting measures for short.
\end{definition}

\begin{proposition} \label{pA.1}
A boundedly finite measure $X$ on $\mathcal{B}^m$ is a counting measure
(i.e., $X \in N$) if and only if
\begin{equation} \label{A.1}
X= \sum_i k_i \, \delta_{x_i},
\end{equation}
where $k_i$ are positive integers and $\{x_i\}$ is a countable set with
at most finitely many $x_i$ in any bounded Borel set.
In Eq. (\ref{A.1}) we use Dirac measures defined for every 
$x_i \in \mathbb{R}^m $ by
\begin{equation} \label{A.2}
\delta_{x_i}(A)=
\begin{cases}
1 & \text{if } x_i \in A, \\
0 & \text{otherwise.}
\end{cases}
\end{equation}
\end{proposition}
We equip $N$ with the $\sigma$-algebra $\mathcal{N}$ generated by sets of the
form
$$\{ X \in N: \, \, X(A)=k \}$$
where $A \in \mathcal{B}^m$ and $k$ is an integer. We finally introduce $N^\ast$ the set of all counting measures such that for all $i\in\N$ it holds $k_i=1$ in \eqref{A.1}.

\begin{definition} \label{dA.2}
A point process $X$ on state space $\mathbb{R}^m$
is a measurable mapping from a probability space
$(\Omega, \mathcal{F}, \mathbb{P})$ into $(N, \mathcal{N})$. 
It is called simple if $X(\omega)\in N^\ast$ a.s..
The distribution of $X$ is the measure $\mu$ on $\mathcal{N}$ induced by
$X$, i.e. 
\begin{equation}\label{eq:def-distribution-PP}
\mu (G)=\mathbb{P} (X^{-1} (G)),
\quad\text{for all }\, G \in \mathcal{N}\,.
\end{equation}
\end{definition}

The notation of Definition \ref{dA.2} is intended to imply that with every
sample point $\omega \in \Omega$, we associate a particular realization
that is a boundedly finite integer-valued Borel measure on $\mathbb{R}^m$.
We denote it by $X(\cdot, \omega)$ or just $X(\cdot)$ (when we have no
need to draw attention to the underlying spaces).
A realization of a point process $X$ has the value $X(A, \omega)$ (or 
just $X(A)$) on the Borel set $A \in \mathcal{B}^m$.
For each fixed $A$, $X_A \equiv X(A, \cdot)$ is a function mapping
$\Omega$ into $\mathbb{R}_+$, and thus it is a candidate for a nonnegative
random variable, as it is shown in the following proposition.

\begin{proposition} \label{pA.2}
Let $X$ be a mapping from a probability space into $N$ and $\mathcal{A}$
a semiring of bounded Borel sets generating $\mathcal{B}^m$.
Then $X$ is a point process if and only if $X_A$ is a random variable
for each $A \in \mathcal{A}$.
\end{proposition}
Taking for $\mathcal{A}$ the semiring of all bounded sets in $\mathcal{B}^m$
we obtain the following corollary.

\begin{corollary} \label{cA.1}
$X: \Omega \mapsto N$ is a point process if and only if $X(A)$ is a random
variable for each bounded $A \in \mathcal{B}^m$.
\end{corollary}

We now consider invariance properties with respect to translations (or
shifts).
Let $T_t$ be the translation in $\mathbb{R}^m$ over the vector $t$:
$T_t(s)=t+s$, for all $s \in \mathbb{R}^m$. Then $T_t$ induces a transformation

$$S_t: N \rightarrow N$$
through the relation
$$(S_t n) (A)=n (T_t^{-1} (A))$$
for all $A \in \mathcal{B}^m$. It is easy to verify that 
$\left(S_t\right)_{t\in\Rm}$ form a group.
\begin{definition} \label{dA.3}
The point process $X$ is said to be stationary if 
\begin{equation}\label{eq:shift-inv-point}
    \forall G\in\mathcal{N}\quad \mu\left(S_t^{-1}(G)\right)=\mu(G)\,.
\end{equation}
\end{definition}
In other words, a process is stationary if for every $A\subset\Rm$, 
the distribution of $n(A)$ is invariant under shifts $t+A$. 
This can be interpreted that $n\in N$ has the same probability as 
all its shifts $S_t n$. 

Since $\P$ induces a probability measure $\mu$ on $(N,\mathcal{N})$ via \eqref{eq:def-distribution-PP}, 
it is convenient to replace the space $(\Omega, \mathcal{F},\P)$
by $(N, \mathcal{N},\mu)$ and to relabel formally 
$(\Omega, \mathcal{F},\P):=(N, \mathcal{N},\mu)$
so that any element $\omega \in \Omega$
represents a counting measure in $\mathbb{R}^m$. 
Identifying $\tau_x:=S_x$, by \eqref{eq:shift-inv-point} we now have a measure-preserving (m. p.) dynamical system 
$(\Omega, \mathcal{F}, \mu, \tau_x)$.

\begin{definition} \label{dA.5}
A stationary point process $\mu$ is said to be ergodic if $\{ \tau_x: x \in 
\mathbb{R}^m \}$ acts ergodically on $(\Omega, \mathcal{F}, \mu)$ 
in the sense of Definition \ref{def:ergodic}.
\end{definition}

\subsection{Percolation theory and random modeling.} \label{appA.2}
The continuum percolation theory provides a general setting for the
realization of random domains.
In this framework, two common models are the Boolean model and the 
random-connection model.

\subsubsection{The Boolean model.} \label{appA.2.1}
The Boolean model is driven by some stationary point process $X$.
Each point of $X$ is the centre of a closed ball (in the usual Euclidean
metric) with a random radius in such a way that radii corresponding to
different points are independent of each other and identically distributed.
The radii are also independent of $X$.
Additionally, we want the resulting random model to be stationary. 
In order to assign independent random values to the radii, 
we partition $\Rm$ into binary cubes
$$K(n,z):=\prod_{i=1}^m [z_i \, 2^{-n}, (z_i+1) \, 2^{-n} ]$$
for all $n \in \mathbb{N}$ and $z \in \mathbb{Z}^m$. 
We call this a binary cube of order $n$.
Each point $x \in X$ is contained in a unique binary cube of order $n$,
$K(n, z(n,x))$ and
for each point $x \in X$ there is a unique smallest number $n_0=n_0(x)$ such
that $K(n_0, z(n_0,x))$ contains no other points of $X$ (recall that $X$ is 
locally finite).
We assign to each point $x_i\in X$ a random value in $[0,\infty)$ in the 
following way: 
For a probability measure $\P_0$ on $[0,\infty)$ we define 
$$\Omega_2:=\prod_{n \in \mathbb{N}} \prod_{z \in \mathbb{Z}^m} [0, \infty)$$
with the corresponding product $\sigma$-algebra and product measure 
$\mathbb{P}_2:=\P_0^{\mathbb{N}\times\mathbb{Z}^m}$.
Denoting by $\omega_2 \in \Omega_2$ the elements of $\Omega_2$ 
we assign to each cube $K(n,z)$ the value $\omega_2(n,z)$ 
and to every $x\in X$ the radius $r=\omega_2(n_0,z(n_0,x))$. 

We now set $\Omega=\Omega_1 \times \Omega_2$ and equip $\Omega$ with product
measure $\mathbb{P}=\mathbb{P}_1 \times \mathbb{P}_2$ and the usual product
$\sigma$-algebra.
A Boolean model is a measurable mapping from $\Omega$ into $N \times \Omega_2$.

The product structure of $\Omega$ implies that the radii are independent
of the point process, and the product structure of $\Omega_2$ implies that
different points have balls with independent, identically distributed radii.

Let the unit vectors in $\mathbb{R}^m$ be denoted by $e_1, \ldots, e_m$.
The translation $T_{e_i}: \mathbb{R}^m \rightarrow \mathbb{R}^m$ defined
by: $x \rightarrow x+e_i$ induces a transformation $U_{e_i}$ on
$\Omega_2$ through the equation

\begin{equation} \label{A.7}
(U_{e_i} \omega_2) (n,z)=\omega_2 (n, z-2^n e_i).
\end{equation}
As before, $S_{e_i}$ is defined on $\Omega_1$ via the equation

\begin{equation} \label{A.8}
(S_{e_i} \omega_1) (A)= \omega_1 (T_{e_i}^{-1} A).
\end{equation}
Hence, $T_{e_i}$ induces a transformation $\tilde{T}_{e_i}$ on
$\Omega=\Omega_1 \times \Omega_2$ defined by

\begin{equation} \label{A.9}
\tilde{T}_{e_i}(\omega)=(S_{e_i} \omega_1, U_{e_i} \omega_2).
\end{equation}
The transformation $\tilde{T}_{e_i}$ corresponds to a translation by the
vector $e_i$ of a configuration of balls in space. The Boolean model  is
now stationary in the sense that $\P$ is shift invariant w.r.t.
$\left(\tilde{T}_x\right)_{x\in\mathbb{Z}^m}$. 
If we replace $\Omega_2$ by $\Omega_2\times[0,1)^m$ as in Sections 2.6 and 3.2 
of \cite{Hei17} we can construct a  family of mappings
$\left(\tau_x\right)_{x\in\mathbb{R}^m}$ on $\Omega$ such that
we have stationarity of $\P$ w.r.t. $\tau_x$.

\subsubsection{The random-connection model.} \label{appA.2.2}
As in Boolean models, a stationary point process $X$ is the first 
characteristic of the random-connection model (RCM) and it assigns
randomly points in the space.
The second characteristic of the model is a so-called connection function,
which is a non-increasing function from the positive reals into $[0, 1]$.
Given a connection function $g$, the rule is as follows: for any two points
$x_1$ and $x_2$ of the point process $X$, we insert an edge between $x_1$
and $x_2$ with probability $g(\vert x_1-x_2 \vert)$, independently of all
other pairs of points of $X$, where $\vert \cdot \vert$ denotes the usual
Euclidean distance.
The formal mathematical construction of a random-connection model is quite
similar to the one of a Boolean model.
First we assume that the point process $X$ is defined on a probability space
$(\Omega_1, \mathcal{F}_1, \mathbb{P}_1)$. 
Next we consider a second probability space $\Omega_2$ defined as

$$\Omega_2=\prod_{\{K(n,z), K(m,z')\}} [0,1]$$
where the product is over all unordered pairs of binary cubes.
An element $\omega_2 \in \Omega_2$ is written as 
$\omega_2(\{(n,z), (m,z') \})$.
We equip $\Omega_2$ with product measure $\mathbb{P}_2$.
As before, we set $\Omega=\Omega_1 \times \Omega_2$ and we equip $\Omega$
with product measure $\mathbb{P}=\mathbb{P}_1 \times \mathbb{P}_2$.
A random-connection model is a measurable mapping from $\Omega$ into
$N \times \Omega_2$ defined by

$$(\omega_1, \omega_2) \rightarrow (X(\omega_1), \omega_2).$$
The realisation corresponding to $(\omega_1, \omega_2)$ is obtained as
follows: for any two points $x$ and $y$ of $X(\omega_1)$, consider the
binary cubes $K(n_0(x), z(n_0(x), x))$ and $K(n_0(y), z(n_0(y), y))$.
We connect $x$ and $y$ if and only if

$$\omega_2(\{ (n_0(x), z(n_0(x), x)), (n_0(y), z(n_0(y), y)) \} ) <
g(\vert x-y \vert).$$
The dynamical system can be constructed similar to the Boolean model.

\subsubsection{The Poisson process.} \label{appA.2.3}
Usually, both the Boolean and the random-connection models are based on
occurrences of the Poisson point process.

\begin{definition} \label{dA.15}
The point process $X$ is said to be a Poisson process with density
$\lambda >0$ if (i) and (ii) below are satisfied:

\par\indent
(i) For mutually disjoint Borel sets $A_1, \ldots, A_k$, the random variables
$X(A_1), \ldots, X(A_k)$ are mutually independent.

\par\indent
(ii) For any bounded Borel set $A \in \mathcal{B}^m$ we have for every
$k \geq 0$

\begin{equation} \label{A.10}
\mathbb{P}(X(A)=k)=e^{-\lambda \, \lebesgueL(A)} \; \; 
\frac{\lambda^k \, \lebesgueL(A)^k}{k!}
\end{equation}
where $\lebesgueL(\cdot)$ denotes Lebesgue measure in $\mathbb{R}^m$.
\end{definition}
Eq. (\ref{A.10}) represents the probability that the number of points inside a
bounded Borel set $A$ equals $k$.
Condition $(ii)$ guarantees that a Poisson process is stationary.
Furthermore, one can prove \cite{MR}:

\begin{proposition} \label{pA.15}
A Poisson point process is ergodic.
\end{proposition}

The following result shows that ergodicity of a point process carries over
to a Boolean model or to a random-connection model driven by that process
\cite{MR}.

\begin{proposition} \label{pA.16}
Suppose $X$ is ergodic. Then, any Boolean model or random-connection model
driven by $X$ is also ergodic.
\end{proposition}

\subsection{Realization of random perforated structures.} \label{appA.3}
In Section \ref{sec:rand-meas-an-sets}, we have stated the main assumptions 
that our perforated
domain should satisfy.
We have also stressed how a random spherical structure can provide a rather
realistic description of neurons in the cerebral tissue.
Unfortunately, a Boolean model driven by the Poisson point process allows,
in general, the perforations (i.e. the balls) to be generated arbitrarily
close to each other so as to form large connected clusters and small angles.
In this case {\bf{Assumption \ref{1}} } no longer holds and our method fails.

One way to construct domains in which the balls are non-intersecting and
have a minimal positive distance between them is to combine the Boolean and
the random-connection model as follows. This procedure is known as Matern process (see \cite{DV}, Example 10.4(d)). 
Let us consider a random-connection model driven by a Poisson process and
applied on a bounded region of $\mathbb{R}^m$. Two points are connected with 
probability 1 if they have distance less than some constant $d_0$. 
All connected points are then deleted from the process. In case of the Poisson 
process this means that a point is deleted with probability $1-\exp(-\lambda d_0)$, 
where $\lambda$ is the intensity of the point process.
Every remaining point will be assigned as the center of a ball
of random radius $\rho(\omega)<\frac{d_0}{2}$.
For simplicity, in our analysis we will consider balls with the same constant
radius $r_0<\frac{d_0}{2}$.
According to this construction, we obtain a domain randomly perforated with
balls of the same radius and with minimal distance between them, which
satisfies all the assumptions stated in Section \ref{sec:rand-meas-an-sets}.
In particular, let $G(\omega)$ be the union of such random spheres, then our
randomly perforated domain can be defined as  

\begin{equation} \label{A.20}
Q(\omega)=\mathbb{R}^m \setminus G(\omega).
\end{equation}

\vskip2mm

\par{\bf Acknowledgement.}\,
B.F.  is supported by the University of Bologna, 
funds for selected research topics,  by MAnET Marie Curie
Initial Training Network, by 
GNAMPA of INdAM (Istituto Nazionale di Alta Matematica ``F. Severi''), Italy, 
and by PRIN of the MIUR, Italy.

M.H. is financed by Deutsche Forschungsgemeinschaft (DFG) through Grant 
CRC 1114
''Scaling Cascades in Complex Systems'', Project C05 {\em Effective models for 
materials and interfaces with multiple scales}.

S.L. is supported by GNFM of INdAM (Istituto Nazionale di Alta 
Matematica ``F. Severi''), Italy.

\end{document}